\documentclass[a4paper,11pt]{article}
\usepackage[top=3.2cm, bottom=3cm, left=3.45cm, right=3.45cm]{geometry}
\usepackage[T1]{fontenc}
\usepackage[utf8]{inputenc}
\usepackage{latexsym}
\usepackage{amsmath,amsthm,amssymb,amsfonts}
\usepackage{mathtools}
\usepackage[english]{babel}
\usepackage{tikz}
\usetikzlibrary{patterns, matrix}
\tikzstyle{NE-lines}=[pattern=north east lines, pattern color=black!45]
\usepackage{multirow}
\usepackage{booktabs}
\usepackage{scalerel}

\DeclareMathOperator{\img}{Im}

\newcommand{\conj}[1]{#1^{\star}}
\newcommand{\basis}{B}                    
\newcommand{\map}{\varphi}                
\newcommand{\alphas}{\alpha_n(S)}         
\newcommand{\betasw}{\kappa_n(S)}         
\newcommand{\betass}{\kappa^{\circ}_n(S)}        
\newcommand{\quand}{\;\;\,\text{and}\,\;\;}

\newcommand{\twoeul}{A}
\newcommand{\twocay}{B}
\newcommand{\twocayhat}{\hat{B}}
\newcommand{\ones}{\mathbf{1}}
\newcommand{\Cay}{\mathrm{Cay}}           
\newcommand{\Sym}{\mathrm{Sym}}           
\newcommand{\nat}{\mathbb{N}}             
\newcommand{\WI}{\mathrm{I}}              
\newcommand{\WIbin}{\WI^{01}}             
\newcommand{\WIgen}{\mathfrak{I}}         
\newcommand{\WIgenbin}{\WIgen^{01}}       
\newcommand{\Asc}{A}                      
\newcommand{\SAsc}{A^{\circ}}                    
\newcommand{\Des}{D}                      
\newcommand{\SDes}{D^{\circ}}                    
\newcommand{\asc}{\mathrm{asc}}           
\newcommand{\des}{\mathrm{des}}           
\newcommand{\id}{\mathrm{id}}

\newcommand{\sort}{\mathrm{sort}}         
\newcommand{\comp}{\mathrm{Comp}}         
\newcommand{\compbin}{\comp^{01}}         
\newcommand{\hproj}[1]{#1\cdot\ones}  

\newcommand{\bang}{\raisebox{-0.3ex}{\scaleto{!}{1.3ex}}}
\newcommand{\Bur}{\mathrm{Bur}}           
\newcommand{\BBur}{\Bur^{01}}             
\newcommand{\PBur}{\Bur^{\bang}}          
\newcommand{\Genbur}{\mathfrak{Bur}}      
\newcommand{\BGenbur}{\Genbur^{01}}
\newcommand{\PGenbur}{\Genbur^{\bang}}
\newcommand{\rows}{\mathrm{row}}
\newcommand{\cols}{\mathrm{col}}
\newcommand{\Mat}{\mathrm{Mat}}           
\newcommand{\BMat}{\Mat^{01}}             
\newcommand{\PMat}{\Mat^{\bang}}          
\newcommand{\Genmat}{\mathfrak{Mat}}      
\newcommand{\BGenmat}{\Genmat^{01}}
\newcommand{\PGenmat}{\Genmat^{\bang}}

\newcommand{\Bal}{\mathrm{Bal}}

\def\multiset#1#2{\ensuremath{\left(\kern-.3em\left(\genfrac{}{}{0pt}{}{#1}{#2}\right)\kern-.3em\right)}}

\newtheorem{theorem}{Theorem}[section]
\newtheorem{proposition}[theorem]{Proposition}
\newtheorem{lemma}[theorem]{Lemma}
\newtheorem{corollary}[theorem]{Corollary}

\newtheorem*{openproblem*}{Open Problem}
\newtheorem{conjecture}[theorem]{Conjecture}
\theoremstyle{definition}
\newtheorem{definition}[theorem]{Definition}
\newtheorem{remark}[theorem]{Remark}
\newtheorem*{remark*}{Remark}
\newtheorem{example}[theorem]{Example}
\newtheorem*{example*}{Example}

\title{Caylerian polynomials}
\author{Giulio Cerbai
\and
Anders Claesson}
\date{}

\begin{document}

\maketitle

\begin{abstract}
  The Eulerian polynomials enumerate permutations according to their number of
  descents.  We initiate the study of descent polynomials over Cayley
  permutations, which we call Caylerian polynomials.  Some classical
  results are generalized by linking Caylerian polynomials to Burge
  words and Burge matrices. The $\gamma$-nonnegativity of the two-sided
  Eulerian polynomials is reformulated in terms of Burge
  structures. Finally, Cayley permutations with a prescribed ascent set
  are shown to be counted by Burge matrices with fixed row sums.
\end{abstract}

\thispagestyle{empty}

\section{Introduction}

Let $\Sym[n]$ denote the set of permutations
of $[n]=\{1,2,\dots,n\}$. A \emph{descent} of a permutation
$v=v_1\cdots v_n\in\Sym[n]$ is an index $i\in[n-1]$
such that $v_i>v_{i+1}$, and we write $\des(v)$ to denote the total number of
descents of $v$. The $n$th \emph{Eulerian polynomial} is defined by
$$
A_n(t)=\sum_{v\in\Sym[n]}t^{\des(v)}
=\sum_{i=0}^{n-1} A(n,i)t^i,
$$
where $A(n,i)$ is the number of permutations of $[n]$ with $i$ descents.
The coefficients $A(n,i)$ are known as the \emph{Eulerian numbers}, and
a permutation statistic that is equidistributed with $\des$ is said to
be Eulerian; a well known example---due to MacMahon---is the number of
exceedances.  The study of Eulerian polynomials and numbers dates back
to the 1755 book \emph{Institutiones calculi differentialis}~\cite{E} by
Leonard Euler. His definition was not the familiar combinatorial one
given above. Rather, Euler defined his polynomials recursively by $A_0(t)=1$
and
\[
  A_n(t) = \sum_{k=0}^{n-1}\binom{n}{k} A_k(t)(t-1)^{n-1-k}.
\]
It was MacMahon~\cite{MaM} who initiated the
study of descent sets of permutations and later Carlitz and
Riordan~\cite{CR} pointed out the relation between descent sets
and Eulerian numbers.
Two alternative definitions of the Eulerian polynomials are given by
the identities
\[
  \frac{tA_n(t)}{(1-t)^{n+1}}
  = \sum_{m \geq 1} m^nt^m\quad\text{and}\quad
  m^n = \sum_{i=0}^{n-1}A(n,i)\binom{m+i}{i},
\]
the first due to Carlitz~\cite{C2} and the second due to Worpitzky~\cite{W}.

It is well known that the $n$th Eulerian polynomial is symmetric and
unimodal~\cite{FS}. A consequence is that it can be expressed as
$$
A_n(t)=\sum_{j}\gamma_{n,j}t^j(1+t)^{n-1-2j},
$$
for suitable coefficients $\gamma_{n,j}$. A beautiful combinatorial
proof that the coefficients $\gamma_{n,j}$ are nonnegative was given
by Foata and Strehl~\cite{FSt}. The proof is based on an action
on $\Sym[n]$ known as \emph{valley-hopping}. It induces a partition
of $\Sym[n]$, and the contribution of each class to the Eulerian polynomial
is equal to $t^j(1+t)^{n-1-2j}$, where $j-1$ is the number of valleys
of any permutation in the class. As a result, the coefficient
$\gamma_{n,j}$ counts the number of equivalence classes with
$j-1$ valleys, and must be nonnegative.
A broader introduction to this topic can be found in the book \textit{Eulerian numbers}~\cite{Pe2} by Petersen.
The same author wrote a short survey~\cite{Pe} that covers the two-sided
Eulerian numbers and the $\gamma$-nonnegativity of the Eulerian polynomials.

In 1978, Gessel and Stanley~\cite{GS} introduced Stirling permutations;
they are permutations of the multiset $\{1^2,2^2,\dots,n^2\}$ such that all
the entries between two copies of the same integer~$i$ are greater than~$i$.
Using the language of permutation patterns, this restriction translates to
avoiding the pattern $212$.
The rationale for their name is that a formula analogous to Carlitz identity
is obtained by replacing the Eulerian polynomial with the descent
polynomial over Stirling permutations, and replacing $m^n$ with the $(m+n,n)$th
Stirling number of the second kind. Several variants of Stirling
permutations have been studied.
Quasi-Stirling permutations~\cite{AGPS} are permutations of the multiset
$\{1^2,2^2,\dots,n^2\}$ that avoid $1212$ and $2121$.
In a note written in 1978 and published in 2020, Gessel~\cite{G}
considers Stirling permutations (in the sense of avoiding 212) of
the multiset $\{1^k,\dots,n^k\}$, for any $k\geq 1$; his
note also contains an up-to-date list of (some of) the literature devoted
to Stirling permutations.
Stirling and quasi-Stirling permutations of arbitrary multisets have been
studied by Brenti~\cite{Br} in the context of Hilbert polynomials, and
by Kuba and Panholzer~\cite{KP} in relation to increasing trees.

In this paper we explore descent polynomials over so called Cayley
permutations. They are permutations of
multisets (without any Stirling or quasi-Stirling restrictions), where
each element has positive multiplicity.
There are known links between Eulerian polynomials and Cayley permutations.
For instance, it is well known (see e.g.\ Stanley~\cite{St}) that the $n$th Eulerian
polynomial evaluated at $2$ is equal to the number of Cayley permutations
of length $n$, which is the $n$th Fubini number.
Furthermore, permutations of length $n$ whose descent set is a subset
of $S=\{s_1,\dots,s_k\}$, with $1\le s_1<s_2<\dots<s_k\le n-1$, are
enumerated by a multinomial coefficient that counts Cayley permutations
containing $s_1$ copies of $1$, $s_2-s_1$ copies of $2$, and so on.
One might say (somewhat imprecisely) that
descents over permutations are governed
by Cayley permutations. A natural question is then what combinatorial
structures govern descents over Cayley permutations.
Here we give an answer in terms of Burge
matrices and Burge words, providing a more general framework
from which most of the classical results for permutations
can be derived as special cases.

The rest of this article is structured as follows.  In
Section~\ref{section_prel} we provide necessary preliminaries. For
instance, we define Cayley permutations
and the related symmetries.  We also introduce Burge matrices and
Burge words. The former are matrices with
nonnegative integer entries whose every row and column contains
at least one nonzero entry.
The latter are biwords where the bottom row is a Cayley permutation
and the top row is a weakly increasing Cayley permutation whose
descent set is a subset of the descent set of the bottom row.

Section~\ref{section_binary} is focused on binary Burge matrices and the
even more restrictive subset of matrices whose every column sums to 1.
These two classes of matrices will later be seen to have close
connections to the Caylerian polynomials. We also make the following
curious observation: By reversing the set inequality that defines Burge
words we obtain a set of biwords that is equinumerous with binary Burge
matrices.

In Section~\ref{section_emptyrows}, we extend the correspondence between
Burge matrices and Burge words by allowing matrices to have empty rows.

In Section~\ref{section_caylpol} we study the weak and strict
descent polynomials over the set of Cayley permutations, which we
call Caylerian polynomials. In Theorem~\ref{thm_caylerian_at_two},
we show that the weak and strict $n$th Caylerian polynomials
evaluated at~$2$ count Burge matrices and binary Burge
matrices, respectively.

In Section~\ref{section_2sidedpol} we define a two-sided version
of the Caylerian polynomials. The main result of this section,
Theorem~\ref{Thm:twosided-caylerians}, highlights a surprising
link with its classical counterpart. As a consequence, we are able
to reformulate the $\gamma$-nonnegativity of the two-sided Eulerian
polynomials in terms of Burge structures.

In Section~\ref{section_comp_maps} we rephrase Stanley's notion
of $v$-compatible maps and generalize it to Cayley permutations.
In Corollary~\ref{cor_fish_basis}, we show that the Fishburn
basis~\cite{CC} of a permutation~$v$ is the set of Cayley
permutations that are compatible with~$v$.
Two equations relating the Caylerian polynomials to
Burge matrices in which empty rows are allowed are obtained
in Theorem~\ref{Thm:eqs-cayler}.

In Section~\ref{section_alphas} we study Cayley permutations
with a prescribed set of ascents and relate them, in
Theorem~\ref{Thm:prescr-ascset}, to Burge matrices
with fixed row sums.

Finally, in Section~\ref{section_final}, we suggest some directions for future work.

\section{Preliminaries}\label{section_prel}

A \emph{Cayley permutation}~\cite{Ca,CC,MF84} on $[n]$ is a mapping
$v:[n]\to[n]$ such that $\img(v)=[k]$ for some $k\leq n$.  We often
identify $v$ with the word $v=v_1\cdots v_n$, where $v_i=v(i)$. Let
$\Cay[n]$ be the set of Cayley permutations on $[n]$.
For instance,
$\Cay[1]=\{ 1\}$,
$\Cay[2]=\{ 11,12,21 \}$ and
$$
\Cay[3]=\{111,112,121,122,123,132,211,212,213,221,231,312,321\}.
$$
A \emph{ballot}, also called an ordered set partition, of $[n]$ is a list of
disjoint blocks (nonempty sets) $B_1B_2\cdots B_k$ whose union is $[n]$.
Let $\Bal[n]$ be the set of ballots of $[n]$.  There is a well-known
one-to-one correspondence between ballots $B_1B_2\cdots B_k$ and Cayley
permutations $v$ where $i\in B_{v(i)}$.  For instance,
$\{ 2,3,5\}\{ 6\}\{ 1,7\}\{ 4\}$ in $\Bal[7]$ corresponds to
$3114123$ in $\Cay[7]$. In particular, $|\Cay[n]|$ is equal to the $n$th
Fubini number (A000670 in the OEIS~\cite{Sl}).

A bijective map $v:[n]\to[n]$ is a \emph{permutation}. In other words,
a permutation is a Cayley permutation where $\img(v)=[n]$.
Let $\Sym[n]$ be the set of permutations on $[n]$ and
note that $\Sym[n]\subseteq\Cay[n]$. Denote by $\id_n$ the
identity permutation $\id_n(i)=i$.
In fact, we shall often omit the subscript in $\id_n$ and let $n$ be
inferred by context.
Moreover, let $\WI[n]$ be the subset of $\Cay[n]$ consisting of the weakly
increasing Cayley permutations:
\[
  \WI[n] = \{v\in\Cay[n]: v(1)\leq v(2)\leq \dots \leq v(n)\}.
\]
Note that $|\WI[n]|=2^{n-1}$ for each $n\ge 1$. For example,
\[
  \WI[1]=\{ 1\},\; \WI[2]=\{ 11,12\}\;\text{ and }\;
\WI[3]=\{ 111,112,122,123\}.
\]

Given a Cayley permutation $v\in\Cay[n]$, define the set of
\emph{descents} and \emph{strict descents} of $v$ as,
respectively,
\begin{align*}
  \Des(v) &= \{ i\in[n-1]: v(i)\ge v(i+1)\}; \\
  \Des^{\circ}(v) &=\{ i\in[n-1]: v(i)> v(i+1)\}.
\end{align*}
Let $\des(v)=|\Des(v)|$ and $\des^{\circ}(v)=|\Des^{\circ}(v)|$.
The sets $\Asc(v)$ and $\Asc^{\circ}(v)$ of \emph{ascents} and
\emph{strict ascents}, as well as their cardinalities $\asc(v)$
and $\asc^{\circ}(v)$, are defined analogously. To avoid confusion,
we will sometimes add the word ``weak'' to ascents and descents
to distinguish them from their strict counterpart.

The four symmetries of a rectangle act on the plot of a Cayley
permutation. They are generated by the reverse and complement
operations, which we now define. For any $v\in\Cay[n]$ and $i\in[n]$,
let the \emph{reverse} of $v$ be $v^r(i)=v(n+1-i)$, and let the
\emph{complement} of $v$ be $v^c(i)=\max(v)+1-v(i)$, where
$\max(v)=\max\{v_i: i\in [n]\}$. The following lemma is easy to prove.

\begin{lemma}\label{Lemma:Asc-Des-properties}
  For $v\in\Cay[n]$,
  \begin{gather*}
    \Des(v^c)  = \Asc(v),\quad
    \SDes(v^c) = \SAsc(v),\\
    \Asc(v^c)  = \Des(v),\quad
    \SAsc(v^c) = \SDes(v),\\
    \Des(v^r)  = n-\Asc(v),\quad
    \SDes(v^r) = n-\SAsc(v), \\
    \Asc(v^r)  = n-\Des(v),\quad
    \SAsc(v^r) = n-\SDes(v),
  \end{gather*}
  where $n-X = \{ n-x : x\in X\}$ for any set of numbers $X$.
\end{lemma}

A \emph{Burge matrix}~\cite{CC} is a matrix with nonnegative integer entries
whose every row and column has at least one nonzero entry.
The size of a Burge matrix is the sum of its entries
and we let $\Mat[n]$ denote the set of Burge matrices of size $n$.
For instance, $\Mat[2]$ contains the five matrices
\begin{equation}\label{five-burge-matrices}
  \begin{bmatrix}
    2
  \end{bmatrix},\,
  \begin{bmatrix}
    1 & 1
  \end{bmatrix},\,
  \begin{bmatrix}
    1 \\
    1
  \end{bmatrix},\,
  \begin{bmatrix}
    1 & 0 \\
    0 & 1
  \end{bmatrix},\,
  \begin{bmatrix}
    0 & 1 \\
    1 & 0
  \end{bmatrix}.
\end{equation}
With each matrix $A=(a_{ij})$ in $\Mat[n]$ we associate a biword
$\binom{u}{v}$ of length $n$ as follows. Any column $\binom{i}{j}$
appears $a_{ij}$ times and the columns are sorted in ascending order
with respect to the top entry, breaking ties by sorting in descending
order with respect to the bottom entry.
For instance, the biword corresponding to the matrix
$$
\begin{bmatrix}
  1 & 0 & 2 \\
  0 & 2 & 1 \\
  1 & 1 & 0
\end{bmatrix}
\quad\text{is}\quad
\binom{1\;1\;1\;2\;2\;2\;3\;3}{3\;3\;1\;3\;2\;2\;2\;1}.
$$
To ease notation, we will often write the biword $\binom{u}{v}$ as
a pair $(u,v)$.

Let $\map$ be the map associating each Burge matrix $A$ with the
corresponding biword $(u,v)$, as described above, and let
$\Bur[n]=\map(\Mat[n])$. Note that the number of rows of $A$ is equal
to $\max(u)$; similarly, the number of columns of $A$ is equal to
$\max(v)$. Furthermore, the requirement that each row and column
of $A$ is not null guarantees that both $u$ and $v$ are Cayley
permutations. In particular, $u$ is a weakly increasing
Cayley permutation.
Finally, we have $\Des(u)\subseteq\Des(v)$ due to the procedure
adopted to sort the columns of $(u,v)$. Indeed, we have
$\Des(u)=\{ i: u(i)=u(i+1)\}$ and entries of $v$
that have the same top entry are sorted in weakly decreasing order.
Conversely, any such biword is associated with a
unique Burge matrix; more explicitly, the biword $(u,v)$ is associated
with the matrix $A=(a_{ij})$, of size $\max(u)\times\max(v)$,
where $a_{ij}$ is equal to the number of columns $\binom{i}{j}$
in $(u,v)$.
In fact, the set $\Bur[n]$ can be alternatively defined as
$$
\Bur[n]=
\{ (u,v)\in\WI[n]\times\Cay[n]: \Des(u)\subseteq \Des(v)\},
$$
and the map $\map$ is a bijection between $\Mat[n]$ and
$\Bur[n]$. Elements of $\Bur[n]$ are called \emph{Burge words}.  This
terminology is due to Alexandersson and Uhlin~\cite{AlUh} and the
connection to Burge is with his variant of the RSK
correspondence~\cite{Bu}.  The sequence of cardinalities
$|\Bur[n]|$ is A120733 in the OEIS~\cite{Sl}.

Given $v\in\Cay[n]$, we define
\[
\WI(v)=\{ u\in\WI[n]:\Des(u)\subseteq\Des(v)\}
\]
so that
\[
\Bur[n] = \bigcup_{v\in\Cay[n]}\WI(v)\times\{ v\},
\]
where the union is disjoint.

Let $A=(a_{ij})\in\Mat[n]$ and let $x=\map(A)$ be its
corresponding biword in $\Bur[n]$. It is
straightforward to compute the biword $x^T$ corresponding to the
transpose $A^T=(a_{ji})$ of $A$: turn each column of $x$ upside down
and then sort the columns of the resulting biword as described
previously. Following~\cite{CC}, we shall write
$$
\binom{u}{v}^T = \binom{\sort(v)}{\Gamma(u,v)},
$$
where $\sort(v)$ is obtained by sorting $v$ in weakly increasing order
and
\[
  \Gamma:\Bur[n]\to\Cay[n]
\]
is the map associating $(u,v)$ with the bottom row of $(u,v)^T$.  As a
notable instance of this construction, if $p$ is a permutation, then
$$
\binom{\id}{p}^T = \binom{\id}{p^{-1}}
\quand
\Gamma(\id,p)=p^{-1}.
$$
Now, it is clear that transposition acts as an involution on the set
of Burge matrices, and thus $\Bur[n]$ is closed under transpose too.
Indeed, we~\cite{CC} showed that the transpose operation gives an
alternative characterization of $\Bur[n]$: For any biword
$(u,v)\in\WI[n]\times\Cay[n]$,
$$
\Des(u)\subseteq\Des(v)
\quad\text{if and only if}\quad
((u,v)^T)^T=(u,v).
$$
Note that the Burge transpose extends naturally to any
biword $(u,v)$ where $u$ and $v$ are maps $[n]\to\{1,2,3,\dots\}$ such that
$u$ is weakly increasing and $\Des(u)\subseteq\Des(v)$.
To compute $(u,v)^T$ we flip $(u,v)$ upside down and sort
the columns of the resulting biword according to the same procedure as before.
It is easy to see that the equality $((u,v)^T)^T=(u,v)$
still holds. Indeed it depends solely on the procedure adopted
to sort the columns and on the initial requirement that
$\Des(u)\subseteq\Des(v)$.
The map $\Gamma$ extends in a similar fashion.
For instance,
$$
\biggl(
\binom{1\;1\;2\;4\;4\;7\;8\;8}{3\;3\;1\;5\;1\;5\;6\;1}^{\!T}\,\biggr)^{\!T} \,=\,
\binom{1\;1\;1\;3\;3\;5\;5\;6}{8\;4\;2\;1\;1\;7\;4\;8}^{\!T} \,=\,
\binom{1\;1\;2\;4\;4\;7\;8\;8}{3\;3\;1\;5\;1\;5\;6\;1}
$$
and $\Gamma(11244788,33151561)=84211748$.
Note that $\Gamma(u_1,v)\neq\Gamma(u_2,v)$
if $u_1\neq u_2$. In other words, for any fixed $v\in\Cay[n]$, the map
$u\mapsto \Gamma(u,v)$ is injective on $\WI(v)\times\{ v\}$.

\section{Binary Burge matrices}\label{section_binary}

We shall now define two subsets of $\Mat[n]$ that will be particularly relevant
in this paper. Let $\BMat[n]$ be the set of \emph{binary Burge matrices} of size $n$, that is,
matrices in $\Mat[n]$ with coefficients in $\{0,1\}$. The cardinality
of $\BMat[n]$ is given by A101370 in the OEIS~\cite{Sl}. The four biwords
corresponding to matrices in $\BMat[2]$ are
\begin{equation}\label{eq_bmats}
  \binom{11}{21},\binom{12}{12},\binom{12}{21},\binom{12}{11}.
\end{equation}
Let $\PMat[n]$ be the set of Burge matrices whose every column sums
to $1$; equivalently, whose every column contains precisely one
nonzero entry, which is equal to $1$. Clearly,
$$
\PMat[n]\subseteq \BMat[n]\subseteq \Mat[n].
$$
For instance, among the five Burge matrices of size two listed in
\eqref{five-burge-matrices}, three belong to $\PMat[2]$, namely
$$
\begin{bmatrix}
  1 & 1
\end{bmatrix},\,
\begin{bmatrix}
  1 & 0 \\
  0 & 1
\end{bmatrix}
\,\text{ and }\,
\begin{bmatrix}
  0 & 1 \\
  1 & 0
\end{bmatrix},
$$
whereas the matrix
$\bigl[\begin{smallmatrix}1 \\ 1\end{smallmatrix}\bigr]$
does not, but it is binary.

Let us define the Burge words corresponding to $\BMat[n]$ and
$\PMat[n]$:
\begin{align*}
  \BBur[n] &\,=\, \bigl\{ (u,v)\in\Bur[n]:\Des(u)\subseteq\Des^{\circ}(v) \bigr\};\\
  \PBur[n] &\,=\, \bigl\{ (u,v)\in\Bur[n]: v\in\Sym[n] \bigr\}.
\end{align*}
Let $v\in\Cay[n]$. In analogy with $\WI(v)$, we also define
\[
  \WIbin(v) = \{ u\in\WI[n]: \Des(u)\subseteq\Des^{\circ}(v)\}
\]
so that
\[
  \BBur[n] = \bigcup_{v\in\Cay[n]}\WIbin(v)\times\{ v\},
\]
where the union is disjoint. Note that
\[
  \PBur[n]\,\subseteq\,\BBur[n]\,\subseteq\,\Bur[n].
\]
Indeed, $\BBur[n]\subseteq\Bur[n]$ since $\Des^{\circ}(v)\subseteq\Des(v)$
and $\WIbin(v)\subseteq\WI(v)$. Furthermore, if $v\in\Sym[n]$,
then $\Des(v)=\Des^{\circ}(v)$, from which $\PBur[n]\subseteq\BBur[n]$
immediately follows.

\begin{proposition}\label{prop_map_subsets}
  For $n\ge 0$,
  $$
  \map(\BMat[n])=\BBur[n]
  \quand
  \map(\PMat[n])=\PBur[n].
  $$
  In particular,
  $$
  |\BMat[n]|=|\BBur[n]|
  \quand
  |\PMat[n]|=|\PBur[n]|.
  $$
\end{proposition}
\begin{proof}
  Let $A\in\Mat[n]$ and $\map(A)=(u,v)\in\Bur[n]$.  We shall start by
  proving that $\map(\BMat[n])=\BBur[n]$.  If $A$ is binary, then
  $(u,v)$ has no repeated columns. In particular, if $u(i)=u(i+1)$, then
  $v(i)\neq v(i+1)$. Moreover, $v(i)\ge v(i+1)$ due to the inequality
  $\Des(u)\subseteq\Des(v)$ defining biwords in $\Bur[n]$.  Thus
  $u(i)=u(i+1)$ implies $v(i)>v(i+1)$, and so $\Des(u)\subseteq\Des^{\circ}(v)$
  and $(u,v)\in\BBur[n]$.  Similarly, if $(u,v)\in\BBur[n]$, i.e.\
  $\Des(u)\subseteq\Des^{\circ}(v)$, then there are no repeated columns in
  $(u,v)$ and $A\in\BMat[n]$.  The equality $\map(\PMat[n])=\PBur[n]$
  follows from the following simple observation:
  Each column of $A$ sums to $1$ if and
  only if $v$ contains exactly one copy of each integer, that is,
  $v\in\Sym[n]$ is a permutation.
\end{proof}

We shall shortly define a natural set of biwords (denoted $\Omega[n]$
below) that is equinumerous with binary Burge words. While we find this
connection interesting, we should point out that the machinery we wish to
develop surrounding Cayley polynomials does not depend on it, and in
that sense the reminder of this section can be skipped.

Recall that
\begin{align*}
  \Bur[n] &= \{ (u,v)\in\WI[n]\times\Cay[n]: \Des(u)\subseteq \Des(v)\}. \\
\intertext{What happens if we reverse the set inequality and define}
  \Omega[n] &= \{ (u, v)\in\WI[n]\times\Cay[n]: \Des(u)\supseteq \Des(v) \}?
\end{align*}
For instance, we have
  \begin{align}\label{eq_omegamats}
    \Omega[2] &= \left\{
    \binom{1\,2}{1\,2},\, \binom{1\,1}{2\,1},\,
    \binom{1\,1}{1\,2},\, \binom{1\,1}{1\,1}
    \right\}.
  \end{align}
Curiously, it turns out that $\Omega[n]$ and $\BBur[n]$ are equinumerous. This
is due to a symmetry on $\WI[n]$ that we now define.
Any $u\in \WI[n]$ is determined by its descent set $\Des(u)$. Indeed,
if we know that $\Des(u)=S\subseteq[n-1]$ and $u\in\WI[n]$, with
$n\geq 1$, then $u$ is given by $u(1)=1$ and $u(i+1)=u(i)+[i\notin S]$.
Here we are using the Iverson bracket, so that $[i\notin S]$ is $1$ if
$i\notin S$ and $0$ if $i\in S$. We define $\conj{u}$ as the unique weakly
increasing Cayley permutation with descent set $[n-1]\setminus D(u)$,
and we call $\conj{u}$ the \emph{conjugate} of $u$. For instance, if
$u=12223445555$, then $\conj{u} = 11233344567$, where
$\Des(\conj{u})=\{1,2,\dots,10\}\setminus\Des(u)=\{1,4,5,7\}$.
Note that if we add the $i$th letter of $u$ with the $i$th letter of
$\conj{u}$, for $i=1$ through $i=11$, we get $2$,$3$,\dots,$12$. This
is not a coincidence. In fact, we have the following lemma, which could
serve as the definition of $\conj{u}$.

\begin{lemma}
  Let $u\in \WI[n]$. For each $i\in[n]$ we have $\conj{u}(i) = i+1-u(i)$.
\end{lemma}
\begin{proof}
   Since $u(1) = \conj{u}(1) = 1$, the statement trivially holds for
   $i=1$. Let $S=\Des(u)$ and $T=\Des(\conj{u})$, so that
   $T = [n-1]\setminus S$. For $i\geq 1$ we have
   \begin{align*}
     \conj{u}(i+1)
     = \conj{u}(i) + [i\notin T]
     &= i+1-u(i) + [i\notin T] \\
     &= i+1-u(i) + (1 - [i\in T]) \\
     &= i+2-(u(i) + [i\notin S]) = i+2 - u(i+1),
   \end{align*}
   which---by induction---concludes the proof.
 \end{proof}
%

The following lemma is easy to prove.

\begin{lemma}\label{Lemma:Asc-Des-conjugation}
  For $u\in\WI[n]$ we have $\Des(\conj{u}) = \SAsc(u)$ and
  $\SAsc(\conj{u}) = \Des(u)$.
\end{lemma}

\begin{proposition}
  Define the mapping $\theta:\BBur[n]\to \Omega[n]$ by
  $\theta(u,v) = \bigl( (u^{rc})^{\star}, v^r \bigr)$. Then $\theta$ is a
  bijection.
\end{proposition}

\begin{proof}
  Note that if $u\in \WI[n]$, then $u^{rc}\in \WI[n]$, and thus
  $(u^{rc})^{\star}$ is well defined.  Since reverse, complement and
  conjugation are involutions, it is clear that $\theta$ is
  injective. It only remains to show that the image of $\BBur[n]$ under
  $\theta$ is $\Omega[n]$. By Lemma~\ref{Lemma:Asc-Des-properties}
  and Lemma~\ref{Lemma:Asc-Des-conjugation} we have
  $$
    \Des\bigl((u^{rc})^{\star}\bigr)
    = \SAsc(u^{rc})
    = \SDes(u^r)
    = n - \SAsc(u)
  \quad\text{and}\quad
  \Des(v^r) = n - \Asc(v).
  $$
  Since $\SAsc(u)\cup\Des(u)=[n-1] = \Asc(v)\cup\SDes(v)$ it follows that
  \begin{align*}
    \theta(u,v) \in \Omega[n]
    &\iff \SAsc(u)\supseteq \Asc(v) \\
    &\iff \Des(u) \subseteq \SDes(v)
    \iff (u,v) \in \BBur[n],
  \end{align*}
  which concludes the proof.
\end{proof}

With the elements in the order they are listed in Equation~\eqref{eq_bmats}
and Equation~\eqref{eq_omegamats}, $\theta$ maps the $i$th element
of $\BBur[2]$ to the $i$th element of $\Omega[2]$.


\section{Allowing empty rows}\label{section_emptyrows}

We wish to generalize our setting by allowing Burge matrices to have
empty rows, the reason for which will become clear in the coming sections.
Let $\Genmat$ consist of all matrices with nonnegative integer entries
whose every column has at least one nonzero entry, and let $\Genmat_m$
be the set of matrices in $\Genmat$ with $m$ rows.  Also, denote by
$\Genmat[n]$ and $\Genmat_m[n]$ the corresponding sets of matrices whose
entries sum to $n$. Clearly, $\Mat[n]\subseteq\Genmat[n]$.
Matrices in $\Genmat_{m}[n]$
have been studied by Munarini, Poneti and Rinaldi~\cite{MPR} under the
name of \emph{$m$-compositions} of $n$.
It is easy to see that the procedure defining
the map $\map$ applies to matrices in $\Genmat$ as well.
However, since empty rows are now admitted, the top row of the
resulting biwords is a weakly increasing map, but not necessarily a
Cayley permutation.
For instance, the biword associated with
$$
\begin{bmatrix}
  1 & 0 & 2 \\
  0 & 0 & 0 \\
  1 & 1 & 0
\end{bmatrix}
\quad\text{is}\quad
\binom{1\;1\;1\;3\;3}{3\;3\;1\;2\;1}.
$$
Another consequence of allowing empty rows is that the correspondence
between matrices and biwords is not bijective anymore. Returning
to the matrix of the previous example, all the matrices obtained
by inserting any number of empty rows at the bottom give rise
to the same biword:
$$
\map^{-1}\binom{1\;1\;1\;3\;3}{3\;3\;1\;2\;1}=
\left\{
\begin{bmatrix}
  1 & 0 & 2 \\
  0 & 0 & 0 \\
  1 & 1 & 0
\end{bmatrix},
\begin{bmatrix}
  1 & 0 & 2 \\
  0 & 0 & 0 \\
  1 & 1 & 0 \\
  0 & 0 & 0
\end{bmatrix},
\begin{bmatrix}
  1 & 0 & 2 \\
  0 & 0 & 0 \\
  1 & 1 & 0 \\
  0 & 0 & 0 \\
  0 & 0 & 0
\end{bmatrix},\dots
\right\}.
$$
Here the ambiguity lies in that we did not specify the codomain of
top row $u$ of the biword $(u,v)=(11133,33121)$.
It is clear that, if we regard $u$ as a mapping $u:[n]\to[m]$, for a fixed
value of $m$, then there is a unique matrix $A$ with $m$ rows 
associated with $(u,v)$. Note also that $m\ge\asc(v)+1$
because of the inclusion $\Des(u)\subseteq\Des(v)$.
Define the set
$$
\WIgen_m[n]=
\{ u: [n]\to[m] \,\;\&\;\, u(1) \leq u(2) \leq\cdots\leq u(n) \}
$$
of weakly increasing maps from $[n]$ to $[m]$, and let
$$
\Genbur_{m}[n]=
\{(u,v)\in\WIgen_m[n]\times\Cay[n]: \Des(u)\subseteq\Des(v)\}.
$$

To ease notation, let $\Bur=\cup_{n\ge 0}\Bur[n]$,
$\Genbur=\cup_{n\ge 0}\Genbur[n]$,
$\Genbur_m=\cup_{n\ge 0}\Genbur_m[n]$, etc. We may also write
$(u,v)\in\WI\times\Cay$ instead of $(u,v)\in\WI[n]\times\Cay[n]$, and
$(u,v)\in\WIgen_m\times\Cay$ instead of
$(u,v)\in\WIgen_m[n]\times\Cay[n]$.  This should not lead to any
confusion since $n$ is determined by the length of $u$, or $v$, and in
all biwords $(u,v)$ considered in this paper $u$ and $v$ will be of
equal length. The correspondence between matrices and biwords described
above induces a family of bijections
$$
\map_{m}:\Genmat_{m}\to\Genbur_{m}.
$$
For instance, if we regard $u=11133$ as a map
$u:[5]\to[3]$, then
$$
\binom{1\;1\;1\;3\;3}{3\;3\;1\;2\;1}\in\Genbur_{3}
\quand
\map_3^{-1}\binom{1\;1\;1\;3\;3}{3\;3\;1\;2\;1}=
\begin{bmatrix}
  1 & 0 & 2 \\
  0 & 0 & 0 \\
  1 & 1 & 0
\end{bmatrix}\in\Genmat_{3}.
$$
On the other hand, if we regard $u$ as a map $u:[5]\to[4]$, then
$$
\binom{1\;1\;1\;3\;3}{3\;3\;1\;2\;1}\in\Genbur_{4}
\quand
\map_4^{-1}\binom{1\;1\;1\;3\;3}{3\;3\;1\;2\;1}=
\begin{bmatrix}
  1 & 0 & 2 \\
  0 & 0 & 0 \\
  1 & 1 & 0 \\
  0 & 0 & 0
\end{bmatrix}\in\Genmat_{4}.
$$
As a further illustration, the biwords $(u,v)\in\Genbur_3[2]$ in which $v=11$,
and the corresponding matrices in $\Genmat_3[2]$, are listed below:
\begin{align*}
  \binom{11}{11}\leftrightarrow
  \begin{bmatrix}2\\0\\0\end{bmatrix}
  \qquad
  \binom{12}{11}\leftrightarrow
  \begin{bmatrix}1\\1\\0\end{bmatrix}
  \qquad
  \binom{13}{11}\leftrightarrow
  \begin{bmatrix}1\\0\\1\end{bmatrix}&\\
  \binom{22}{11}\leftrightarrow
  \begin{bmatrix}0\\2\\0\end{bmatrix}
  \qquad
  \binom{23}{11}\leftrightarrow
  \begin{bmatrix}0\\1\\1\end{bmatrix}
  \qquad
  \binom{33}{11}\leftrightarrow
  \begin{bmatrix}0\\0\\2\end{bmatrix}&.
\end{align*}
We define subsets of $\Genmat$ and $\Genbur$ that are analogous to
$\BMat$, $\PMat$, $\BBur$ and $\PBur$ as follows.
Let $\BGenmat_{m}$ be the set of binary matrices in $\Genmat_{m}$
and let $\PGenmat_{m}$ be the set of matrices in $\Genmat_{m}$
where each column sums to $1$. Furthermore, let
\begin{align*}
  \BGenbur_{m} &= \{(u,v)\in\Genbur_{m}:\Des(u)\subseteq\Des^{\circ}(v)\} \\
  \shortintertext{and}
  \PGenbur_{m} &= \{(u,v)\in\Genbur_{m}:v\in\Sym\}.
\end{align*}
Clearly,
$\PGenmat_{m}\subseteq\BGenmat_{m}\subseteq\Genmat_{m}$
and
$\PGenbur_{m}\subseteq\BGenbur_{m}\subseteq\Genbur_{m}$.
Given $v\in\Cay[n]$, let
\begin{align*}
\WIgen_m(v)=\{ u\in\WIgen_m[n]: \Des(u)\subseteq\Des(v)\}
&\quand
\WIgenbin_m(v)=\{ u\in\WIgen_m[n]: \Des(u)\subseteq\Des^{\circ}(v)\}
\intertext{so that we have the disjoint unions}
\Genbur_{m}=\bigcup_{v\in\Cay}\WIgen_m(v)\times\{ v\}
&\quand
\BGenbur_{m}=\bigcup_{v\in\Cay}\WIgenbin_m(v)\times\{ v\}.
\end{align*}

The proof of the following result is akin to the proof
of Proposition~\ref{prop_map_subsets} and is omitted.

\begin{proposition}\label{prop_genmat_genbur}
  For each $m\ge 0$,
  $$
  \map_m(\BGenmat_m)=\BGenbur_m
  \quand
  \map_m(\PGenmat_m)=\PGenbur_m.
  $$
  In particular, for each $n\ge 0$,
  $$
  |\BGenmat_m[n]|=|\BGenbur_m[n]|
  \quand
  |\PGenmat_m[n]|=|\PGenbur_m[n]|.
  $$
\end{proposition}

\section{Caylerian polynomials}\label{section_caylpol}

Let $n\ge 0$. The \emph{$n$th Eulerian polynomial} is
$$
A_n(t)=\sum_{v\in\Sym[n]}t^{\des(v)}.
$$

We shall study the corresponding descent polynomials over Cayley
permutations. There are two reasonable definitions: the \emph{$n$th
(weak)  Caylerian polynomial} and the \emph{$n$th strict Caylerian polynomial}
are defined, respectively, as
$$
C_n(t)=\sum_{v\in\Cay[n]}t^{\des(v)}
\quand
C^{\circ}_n(t)=\sum_{v\in\Cay[n]}t^{\des^{\circ}(v)}.
$$
It is clear (e.g.\ by Lemma~\ref{Lemma:Asc-Des-properties}) that in the
above definitions we can replace $\des(v)$ with $\asc(v)$ and
$\des^{\circ}(v)$ with $\asc^{\circ}(v)$, respectively. In other words, we have
$$
A_n(t)=\sum_{v\in\Sym[n]}t^{\asc(v)},\quad
C_n(t)=\sum_{v\in\Cay[n]}t^{\asc(v)},\quad
C^{\circ}_n(t)=\sum_{v\in\Cay[n]}t^{\asc^{\circ}(v)}.
$$
Furthermore, since $\des^{\circ}(v)=n-1-\asc(v)$ for each $v\in\Cay_n$,
the coefficients of the strict Caylerian polynomial $C^{\circ}_n(t)$ are simply
the reverse of the coefficients of $C_n(t)$:
$$
C^{\circ}_n(t)=t^{n-1}C_n\bigl(1/t\bigr).
$$
The first Caylerian polynomials are
\begin{align*}
  C_0(t)&=1;\\
  C_1(t)&=1;\\
  C_2(t)&=1+2t;\\
  C_3(t)&=1+8t+4t^2;\\
  C_4(t)&=1+24t+42t^2+8t^3;\\
  C_5(t)&=1+64t+276t^2+184t^3+16t^4.
\end{align*}
The resulting triangle of coefficients is A366173 in the OEIS~\cite{Sl}.

Exercise~133(a) of Chapter 1 in Stanley's \emph{Enumerative
Combinatorics}~\cite{St} asks for a combinatorial proof that the $n$th Eulerian
polynomial evaluated at $2$ equals the number of Cayley permutations on $[n]$.
A simple one in terms of the Burge transpose goes as follows~\cite{CC}.
For any permutation $v$, we have $|\WI(v)|=2^{\des(v)}$.
In other words, there are $2^{\des(v)}$ weakly
increasing Cayley permutations $u$ such that $(u,v)\in\PBur$.
Indeed, in order to satisfy the inclusion $\Des(u)\subseteq\Des(v)$,
if $v(i)<v(i+1)$ is an ascent, then we must have $u(i+1)=u(i)+1$.
On the other hand, if $v(i)>v(i+1)$ is a descent, then both $u(i+1)=u(i)$
and $u(i+1)=u(i)+1$ are admitted. Thus,
\begin{align*}
  A_n(2) &\,=\, \sum_{v\in\Sym[n]}2^{\des(v)} \\
         &\,=\, |\{(u,v)\in\WI[n]\times\Sym[n]:u\in\WI(v)\}| \,=\, |\PBur[n]|.
\end{align*}
Furthermore,
\begin{align*}
\PBur[n]^T
  &= \left\{(u,v)\in\WI[n]\times\Sym[n]:u\in\WI(v)\right\}^T \\
  &= \left\{(\id,x):x\in\Cay[n]\right\},
\end{align*}
and clearly the cardinality of the latter set is $|\Cay[n]|$.
The main advantage of this approach is that it works for
Cayley permutations as well. For $v\in\Cay$ there are
$2^{\des(v)}$ weakly increasing Cayley permutations $u\in\WI(v)$
such that $(u,v)\in\Bur$. Therefore,
\begin{align*}
  C_n(2) &\,=\, \sum_{v\in\Cay[n]}\!2^{\des(v)} \\
         &\,=\, \sum_{v\in\Cay[n]}\!|\WI(v)| \,=\, |\Bur[n]|.
\end{align*}

An analogous equality holds for $C^{\circ}_n(2)$. We summarize
the results obtained this way in the following theorem.

\begin{theorem}\label{thm_caylerian_at_two}
  For $n\ge 0$,
  \begin{itemize}
  \item[$(i)$] $C_n(2) = \bigl|\Bur[n]\bigr| = \bigl|\Mat[n]\bigr|;$
  \item[$(ii)$] $C^{\circ}_n(2) = \bigl|\BBur[n]\bigr| = \bigl|\BMat[n]\bigr|;$
  \item[$(iii)$] $A_n(2) = \bigl|\PBur[n]\bigr| = \bigl|\PMat[n]\bigr|.$
  \end{itemize}
\end{theorem}

Let us push this approach a bit further.
Recall from~\cite{CC} that the \emph{Fishburn basis} of $v\in\Sym$
is defined by
$$
\basis(v) = \Gamma\bigl(\WI(v)\times\{ v\}\bigr),
$$
or, equivalently,
$$
\basis(v) = \left\{ x\in\Cay:(\id,x)^T=(\sort(x),v)\right\}.
$$
Once again, it is easy to see that $|\basis(v)|=2^{\des(v)}$.
Now, for each $x\in\Cay$ there is exactly one permutation $v\in\Sym$
such that $x\in\basis(v)$. Indeed, we have
\begin{equation}\label{basis-characterization}
  x\in\basis(v)\iff
  \binom{\id}{x}^T=\binom{\sort(x)}{v}
\end{equation}
and thus $v=\Gamma(\id,x)$.
As a result, the disjoint union
$$
\Cay=\bigcup_{v\in\Sym}\basis(v)
$$
holds. In particular, since $(\id,v)\in\Bur$ and
$(\id,v)^T=(\id,v^{-1})$, the inverse permutation $v^{-1}$
belongs to the Fishburn basis of $v$.
Next we show that every Cayley permutation in $\basis(v)$ has the same
weak descent set (and thus also strict ascent set) as $v^{-1}$.

\begin{lemma}\label{lemma_basis_descents}
  Let $v$ be a permutation. For each $x\in\basis(v)$, we have $\Des(x)=\Des(v^{-1})$
  and $\Asc^{\circ}(x)=\Asc(v^{-1})$. Furthermore,
  $$
  \sum_{x\in\basis(v)}t^{\des(x)}=2^{\des(v)}t^{\des(v^{-1})}
  \quand
  \sum_{x\in\basis(v)}t^{\asc^{\circ}(x)}=2^{\des(v)}t^{\asc(v^{-1})}.
  $$
\end{lemma}
\begin{proof}
  Let $v\in\Sym[n]$, $x\in\basis(v)$ and $u=\sort(x)$, so that
  $(\id,x)^T=(u,v)$. For $i\in[n-1]$, the columns
  $$
  \binom{i}{x_i}\quand\binom{i+1}{x_{i+1}}
  \quad\text{in}\quad\binom{\id}{x}
  $$
  are mapped via the Burge transpose to the columns
  $$
  \binom{x_i}{i}\quand\binom{x_{i+1}}{i+1}
  \quad\text{in}\quad\binom{u}{v}.
  $$
  Because of how the columns of $(\id,x)^T$ are sorted via the Burge
  transpose, $i\in\Des(x)$ if and only if $i+1$ precedes $i$ in $v$. In
  particular, if $i\in\Des(x)$ and $x_i=x_{i+1}$, then $i+1$ precedes $i$
  since the corresponding top entries $x_{i+1}$ and $x_i$ are tied and
  $i+1>i$.  We leave the remaining details to the reader.  Furthermore,
  $i+1$ precedes $i$ in $v$ if and only if $i$ is a descent in $v^{-1}$,
  from which $\Des(x)=\Des(v^{-1})$ follows.  Thus,
  \begin{align*}
  \Asc^{\circ}(x)
    &= [n-1]\setminus\Des(x) \\
    &= [n-1]\setminus\Des(v^{-1}) = \Asc(v^{-1}).
  \end{align*}
  To prove the remaining identities, recall that $v^{-1}\in\basis(v)$
  and $|\basis(v)|=2^{\des(v)}$. Thus,
  \begin{align*}
    \sum_{x\in\basis(v)}t^{\des(x)} &=
    \sum_{x\in\basis(v)}t^{\des(v^{-1})} = 2^{\des(v)}t^{\des(v^{-1})}.
  \end{align*}
  The second identity follows in a similar manner.
\end{proof}

Recall that $\Cay[n]$ is the disjoint union $\Cay[n]=\bigcup_{v\in\Sym[n]}\basis(v)$.
By Lemma~\ref{lemma_basis_descents} we now have
$$
  C_n(t)
  = \sum_{v\in\Sym[n]}\sum_{x\in\basis(v)}t^{\des(x)}=
    \sum_{v\in\Sym[n]}2^{\des(v)}t^{\des(v^{-1})}
$$
and a similar formula holds for the strict Caylerian polynomials.
Let us summarize this in a theorem:

\begin{theorem}\label{Thm:C-and-Cprime}
  For $n\ge 0$,
  $$
  C_n(t)=\sum_{v\in\Sym[n]}2^{\des(v)}t^{\des(v^{-1})}
  \quand
  C^{\circ}_n(t)=\sum_{v\in\Sym[n]}2^{\des(v)}t^{\asc(v^{-1})}.
  $$
\end{theorem}

\section{Two-sided Caylerian polynomials}\label{section_2sidedpol}

The $n$th \emph{two-sided Eulerian polynomial}~\cite{CRS,Pe} is
\[
  \twoeul_n(s,t)=\sum_{v\in\Sym[n]}s^{\des(v)}t^{\des(v^{-1})}.
\]
Note that the identity for the weak Caylerian polynomial $C_n(t)$ in
Theorem~\ref{Thm:C-and-Cprime} can be expressed in terms of these
polynomials as $C_n(t)=\twoeul_n(2,t)$.
We define a strict version of the two-sided Eulerian polynomials by
\[
  \twoeul^{\circ}_n(s,t)=\sum_{v\in\Sym[n]}s^{\des(v)}t^{\asc(v^{-1})},
\]
so that $C^{\circ}_n(t)=\twoeul^{\circ}_n(2,t)$.  The (single-sided) Eulerian
polynomials are related to Cayley permutations by
\begin{equation}\label{eulerian-cay}
  A_n(1+t)  = \sum_{v\in\Cay[n]}t^{n-\max(v)}
\end{equation}
and
\begin{equation}\label{eulerian-egf}
  \sum_{n\ge 0}A_n(1+t)\frac{x^n}{n!} = \frac{t}{1+t-\exp(tx)}.
\end{equation}
For an elegant proof of Equations~\eqref{eulerian-cay} and~\eqref{eulerian-egf}
using weighted $\mathbb{L}$-species see Exercise~(1a) of Section~5.1 in the
book by Bergeron, Labelle and Leroux~\cite{BLL}.
We wish to derive analogous results for descent polynomials over Cayley
permutations and Burge structures. To this end, we define the \emph{weak} and
\emph{strict two-sided Caylerian polynomials} by
\begin{align*}
  \twocay_n(s,t) &= \sum_{(u,v)\in\Bur[n]}s^{n-\max(u)}t^{n-\max(v)}; \\
  \twocay^{\circ}_n(s,t) &= \sum_{(u,v)\in\BBur[n]}s^{n-\max(u)}t^{n-\max(v)}.
\end{align*}
Equivalently---in terms of Burge matrices---we have
\begin{align*}
  \twocay_n(s,t) &= \sum_{A\in\Mat[n]}s^{n-\rows(u)}t^{n-\cols(v)}; \\
  \twocay^{\circ}_n(s,t) &= \sum_{A\in\BMat[n]}s^{n-\rows(u)}t^{n-\cols(v)},
\end{align*}
where $\rows(A)$ and $\cols(A)$ denote the number of rows and columns
of $A$, respectively.
The polynomial $\twocay^{\circ}_n(s,t)$ has been studied by Riordan
and Stein~\cite{RS} in terms of arrangements on chessboards (where each line
contains at least one piece).

\begin{lemma}\label{lemma_oneplust_des}
  For $v\in\Cay[n]$,
  \begin{align*}
    (1+t)^{\des(v)} &= \sum_{u\in\WI(v)}t^{n-\max(u)}
    \shortintertext{and}
    (1+t)^{\des^{\circ}(v)} &= \sum_{u\in\WIbin(v)}t^{n-\max(u)}.
  \end{align*}
\end{lemma}
\begin{proof}
  Let $v\in\Cay[n]$. Also, let $u\in\WI(v)$ and recall that by
  definition we then have $\Des(u)\subseteq\Des(v)$.  If $v(i)<v(i+1)$
  is a strict ascent, then $u(i+1)=u(i)+1$.  Otherwise, if
  $v(i)\ge v(i+1)$ is a weak descent, then we have either $u(i+1)=u(i)$
  or $u(i+1)=u(i)+1$.  Let us mark every such entry $u(i+1)$ with $t$
  if $u(i+1)=u(i)$, and with $1$ if $u(i+1)=u(i)+1$. From
  $$
  n=\max(u)+\{ i\in[n-1]: u(i)=u(i+1)\}
  $$
  it follows that
  $n-\max(u)$ is equal to the number of entries of $u$ marked with $t$.
  By summing the contribution $t^{n-\max(u)}$ of each weakly
  increasing Cayley permutation $u\in\WI(v)$ we obtain $(1+t)^{\des(v)}$.
  The second identity is proved in a similar manner.
\end{proof}
\begin{lemma}\label{lemma_basis_or_wi}
  For $v\in\Sym[n]$,
  $$
  \sum_{x\in\basis(v)}t^{n-\max(x)}=\sum_{u\in\WI(v)}t^{n-\max(u)}.
  $$
\end{lemma}
\begin{proof}
  Let $v\in\Sym[n]$. Recall the characterization of the Fishburn basis
  given in equation~\eqref{basis-characterization}:
  $x\in\basis(v)\Longleftrightarrow (\id,x)^T=(\sort(x),v)$. In particular, the Burge
  transpose bijectively maps
  $$
  \left\lbrace\binom{\id}{x}:x\in\basis(v)\right\rbrace
  \quad\text{to}\quad
  \left\lbrace\binom{u}{v}:u\in\WI(v)\right\rbrace.
  $$
  Since $\max(x)=\max\bigl(\sort(x)\bigr)$, the claimed identity
  immediately follows.
\end{proof}

Let us show how equations~\eqref{eulerian-cay} and~\eqref{eulerian-egf} can be derived
from Lemmas~\ref{lemma_oneplust_des} and~\ref{lemma_basis_or_wi}.
First, equation~\eqref{eulerian-cay}:
\begin{align*}
A_n(1+t)&=\sum_{v\in\Sym[n]}(1+t)^{\des(v)}\\
&=\sum_{v\in\Sym[n]}\sum_{u\in\WI(v)}t^{n-\max(u)} &
\text{(by Lemma~\ref{lemma_oneplust_des})}\\
&=\sum_{v\in\Sym[n]}\sum_{x\in\basis(v)}t^{n-\max(x)} &
\text{(by Lemma~\ref{lemma_basis_or_wi})}\\
&=\sum_{x\in\Cay[n]}t^{n-\max(x)},
\end{align*}
where the last equality follows from
$\Cay[n]=\bigcup_{v\in\Sym[n]}\basis(v)$.
Second, equation~\eqref{eulerian-egf}:
\begin{align*}
\sum_{n\ge 0}A_n(1+t)\frac{x^n}{n!}&=
\sum_{n\ge 0}\sum_{v\in\Sym[n]}(1+t)^{\des(p)}\frac{x^n}{n!}\\
&=\sum_{n\ge 0}\sum_{x\in\Cay[n]}t^{n-\max(x)}\frac{x^n}{n!}\\
&=\frac{t}{1+t-\exp(tx)},
\end{align*}
where the last expression is the exponential generating function
of weighted ballots (i.e.\ Cayley permutations) with weight $n$ minus the
number of blocks (i.e.\ the maximum value). 

\begin{proposition}\label{prop_oneplus_square}
  For $p\in\Sym[n]$,
  $$
  (1+s)^{\des(p)}(1+t)^{\des(p^{-1})}=
  \sum_{x\in\basis(p^{-1})}\sum_{u\in\WI(x)}s^{n-\max(u)}t^{n-\max(x)}.
  $$
\end{proposition}
\begin{proof}
  Recall that $p\in\basis(p^{-1})$. By Lemma~\ref{lemma_basis_descents},
  each Cayley permutation $x\in\basis(p^{-1})$ has the same weak descent set
  $\Des(x)=\Des(p)$ as $p$. In particular, $\WI(x)=\WI(p)$ for each
  $x\in\basis(p^{-1})$. Thus
  \begin{align*}
    \sum_{x\in\basis(p^{-1})}\sum_{u\in\WI(x)}s^{n-\max(u)}t^{n-\max(x)}&=
    \sum_{x\in\basis(p^{-1})}t^{n-\max(x)}\left(\sum_{u\in\WI(p)}s^{n-\max(u)}\right)\\
    &=\sum_{x\in\basis(p^{-1})}t^{n-\max(x)}(1+s)^{\des(p)}\\
    &=(1+s)^{\des(p)}\sum_{v\in\WI(p^{-1})}t^{n-\max(v)}\\
    &=(1+s)^{\des(p)}(1+t)^{\des(p^{-1})},
  \end{align*}
  where the second and the last equalities follow by
  Lemma~\ref{lemma_oneplust_des} and the penultimate equality follows
  from Lemma~\ref{lemma_basis_or_wi}.
\end{proof}

To illustrate the previous result, let $p=2413$ and $q=p^{-1}=3142$.
Then $\des(p)=1$, $\des(q)=2$, and
\begin{align*}
  \WI(q) &= \{ 1122,1233,1123,1234\}; \\
  \basis(q) &= \Gamma\left(\WI(q)\times\{ q\}\right) = \{ 1212,1312,2313,2413\}.
\end{align*}
Note that $p\in\basis(q)$ and $\Des(x)=\Des(p)$ for each $x\in\basis(q)$;
in particular, $\WI(x)=\WI(p)$ for each $x\in\basis(q)$.
Now, we have $\WI(p)=\{ 1223,1234\}$ and
$$
\sum_{u\in\WI(p)}s^{n-\max(u)}=s+1=(1+s)^{\des(p)},
$$
as claimed in Lemma~\ref{lemma_oneplust_des}. Referring to the same lemma,
$u=1223$ is obtained by letting $u(3)=u(2)$, and it thus contributes with $s$
to the summand. On the other hand, the contribution of $u=1234$ is $1$
since $u$ is obtained by letting $u(3)=u(2)+1$. Finally,
\begin{align*}
  \sum_{x\in\basis(q)}\sum_{u\in\WI(x)}s^{n-\max(u)}t^{n-\max(x)}
  &=(1+s)^{\des(p)}\sum_{x\in\basis(q)}t^{n-\max(x)}\\
  &=(1+s)^{\des(p)}(t^2+2t+1)\\
  &=(1+s)^{\des(p)}(1+t)^2\\
  &=(1+s)^{\des(p)}(1+t)^{\des(q)},
\end{align*}
as claimed in Proposition~\ref{prop_oneplus_square}.

\begin{theorem}\label{Thm:twosided-caylerians}
  For $n\ge 0$,
  $$
  \twoeul_n(1+s,1+t)=\twocay_n(s,t)
  \quand
  \twoeul^{\circ}_n(1+s,1+t)=\twocay^{\circ}_n(s,t).
  $$
\end{theorem}
\begin{proof}
  Recall that $\Cay[n]$ is the disjoint union $\Cay[n]=
  \bigcup_{q\in\Sym[n]}\basis(q)$. By letting $q=p^{-1}$, we have
  \begin{align*}
  \Bur[n]
    &= \bigcup_{v\in\Cay[n]}\bigcup_{u\in\WI(v)}\{(u,v)\} \\
    &= \bigcup_{p\in\Sym[n]}\bigcup_{x\in\basis(p^{-1})}\bigcup_{u\in\WI(x)}\{(u,x)\},
  \end{align*}
  where once again all the unions are disjoint.
  The desired equation for $\twoeul(1+s,1+t)$ follows by
  Proposition~\ref{prop_oneplus_square}:
  \begin{align*}
    \twocay_n(s,t)&=\sum_{(u,v)\in\Bur[n]}s^{n-\max(u)}t^{n-\max(v)}\\
                  &=\sum_{p\in\Sym[n]}\sum_{x\in\basis(p^{-1})}\sum_{u\in\WI(x)}s^{n-\max(u)}t^{n-\max(x)}\\
                  &=\sum_{p\in\Sym[n]}(1+s)^{\des(p)}(1+t)^{\des(p)^{-1}}\\
                  &=\twoeul_n(1+s,1+t).
  \end{align*}
  A more involved computation is required for the strict case:
  \begin{align*}
    \twocay^{\circ}_n(s,t)
    &=\sum_{(u,v)\in\BBur[n]}s^{n-\max(u)}t^{n-\max(v)}\\
    &=\sum_{v\in\Cay[n]}\sum_{u\in\WIbin(v)}s^{n-\max(u)}t^{n-\max(v)}\\
    &=\sum_{v\in\Cay[n]}t^{n-\max(v)}(1+s)^{\des^{\circ}(v)}
    & \text{(by Lemma~\ref{lemma_oneplust_des})}\\
    &=\sum_{w\in\Cay[n]}t^{n-\max(w)}(1+s)^{\asc^{\circ}(w)}
    & (w=v^r)\\
    &=\sum_{q\in\Sym[n]}\sum_{w\in\basis(q^{-1})}t^{n-\max(w)}(1+s)^{\asc^{\circ}(w)}\\
    &=\sum_{q\in\Sym[n]}\sum_{w\in\basis(q^{-1})}t^{n-\max(w)}(1+s)^{\asc(q)}
    & \text{(by Lemma~\ref{lemma_basis_descents})}   \\
    &=\sum_{q\in\Sym[n]}(1+s)^{\asc(q)}\left(\sum_{w\in\WI(q^{-1})}t^{n-\max(w)}\right)
    & \text{(by Lemma~\ref{lemma_basis_or_wi})}\\
    &=\sum_{q\in\Sym[n]}(1+s)^{\asc(q)}(1+t)^{\des(q^{-1})}
    & \text{(by Lemma~\ref{lemma_oneplust_des})}\\
    &=\sum_{p\in\Sym[n]}(1+s)^{\des(p)}(1+t)^{\asc(p^{-1})}
    & (p=q^r). & \qedhere
  \end{align*}
\end{proof}

As discussed at the beginning of this section we have
$C_n(t)=\twoeul_n(2,t)$ and $C^{\circ}_n(t)=\twoeul^{\circ}_n(2,t)$. Now,
by Theorem~\ref{Thm:twosided-caylerians}, we obtain the following
equations for the Caylerian polynomials:
$$
C_n(1+t)=\twocay_n(1,t)
\quand
C^{\circ}_n(1+t)=\twocay^{\circ}_n(1,t).
$$
The above equations can be seen as analogs of
equation~\eqref{eulerian-cay}.  The problem of finding an exponential
generating function for $C_n(1+t)$ and $C^{\circ}_n(1+t)$, akin to
equation~\eqref{eulerian-egf}, remains open.

We shall reformulate Theorem~\ref{Thm:twosided-caylerians} in terms of
the joint distribution of $\max(u)$ and $\max(v)$ on Burge words
$(u,v)$, or, equivalently, the joint distribution of $\rows(A)$ and
$\cols(A)$ on Burge matrices $A$.  Let
$$
\twocayhat_n(s,t)=
\sum_{(u,v)\in\Bur[n]}s^{\max(u)}t^{\max(v)}
=\sum_{A\in\Mat[n]}s^{\rows(A)}t^{\cols(A)}.
$$
Then
$$
\twocayhat_n(s,t)
=(st)^n\twocay_n\left(\frac{1}{s},\frac{1}{t}\right)
=(st)^n\twoeul_n\left(\frac{s+1}{s},\frac{t+1}{t}\right);
$$
alternatively, by setting $x=(s+1)/s$ and $y=(t+1)/t$, we get
$$
\frac{\twoeul_n(x,y)}{(x-1)^n(y-1)^n}
=\twocayhat_n\left(\frac{1}{x-1},\frac{1}{y-1}\right).
$$
In terms of the Caylerian polynomials,
$$
C_n(t)=\twocay_n(1,t-1)
=(t-1)^n\twocayhat_n\left(1,\frac{1}{t-1}\right).
$$
Similar results hold for the strict case. We thus obtain the
following result.

\begin{corollary}\label{cor_twohat}
  For $n\ge 0$,
  $$
  C_n(t)=(t-1)^n\twocayhat_n\left(1,\frac{1}{t-1}\right)
  \quand
  C^{\circ}_n(t)=(t-1)^n\twocayhat^{\circ}_n\left(1,\frac{1}{t-1}\right).
  $$
  Furthermore,
  $$
  \twocayhat_n(s,t)
  =(st)^n\twoeul_n\left(\frac{s+1}{s},\frac{t+1}{t}\right)
  \quand
  \twocayhat^{\circ}_n(s,t)
  =(st)^n\twoeul^{\circ}_n\left(\frac{s+1}{s},\frac{t+1}{t}\right).
  $$
\end{corollary}

\begin{remark}
  We can use Theorem~\ref{Thm:twosided-caylerians} and
  Corollary~\ref{cor_twohat} to reinterpret the symmetry
  $C^{\circ}_n(t)=t^{n-1}C_n(1/t)$ in terms of the two-sided Caylerian
  polynomials. A simple computation leads to
  $$
  \twocay^{\circ}_n(1,t-1)
  =t^{n-1}\twocay_n\left(1,\frac{1-t}{t}\right),
  $$
  or, alternatively,
  $$
  \twocayhat^{\circ}_n\left(1,\frac{1}{t-1}\right)
  =\frac{(-1)^n}{t}\twocayhat_n\left(1,\frac{t}{1-t}\right).
  $$
  What is the combinatorial meaning of this equality in terms of
  Burge structures?
\end{remark}

Figure~\ref{figure_joint_distr} shows the polynomials $\twocayhat_n(s,t)$
and $\twocayhat^{\circ}_n(s,t)$ up to $n=5$.
Below we illustrate the equalities of Corollary~\ref{cor_twohat} for $n=2$.
The five matrices in $\Mat[2]$, and the corresponding
contribution $s^{\rows(A)}t^{\cols(A)}$ to $\twocayhat_2(s,t)$, are
\begin{align*}
  \begin{bmatrix} 2 \end{bmatrix}=st,\quad
  \begin{bmatrix} 1 & 1\end{bmatrix}=st^2,\quad
  \begin{bmatrix} 1 \\ 1 \end{bmatrix}=s^2t,\quad
  \begin{bmatrix} 1 & 0 \\ 0 & 1 \end{bmatrix}=s^2t^2,\quad
  \begin{bmatrix} 0 & 1 \\ 1 & 0 \end{bmatrix}=s^2t^2,
\end{align*}
which gives
$$
\twocayhat_2(s,t)=st+st^2+s^2t+2s^2t^2.
$$
It is now easy to check that
\[
(st)^2\twoeul_2\left(\frac{s+1}{s},\frac{t+1}{t}\right)
=\twocayhat_2(s,t),
\]
as expected. The analogous computation for binary matrices and the strict
two-sided Caylerian polynomial gives
$$
\twocayhat^{\circ}_2(s,t)
=st^2+s^2t+2s^2t^2
=(st)^2\twoeul^{\circ}_2\left(\frac{s+1}{s},\frac{t+1}{t}\right).
$$

Finally,
\begin{align*}
(t-1)^2\twocayhat_2\left(1,\frac{1}{t-1}\right)
&=2t+1=C_2(t)
\shortintertext{and}
(t-1)^2\twocayhat^{\circ}_2\left(1,\frac{1}{t-1}\right)
&=t+2
=C_2^{\circ}(t).
\end{align*}

\begin{figure}
  \begin{align*}
    \begin{bmatrix}
      1
    \end{bmatrix},\quad
    \begin{bmatrix}
      1 & 1\\
      1 & 2
    \end{bmatrix},\quad
    \begin{bmatrix}
      1 & 2 & 1\\
      2 & 8 & 6\\
      1 & 6 & 6
    \end{bmatrix},\quad
    \begin{bmatrix}
      1 & 3 & 3 & 1\\
      3 & 19 & 30 & 14\\
      3 & 30 & 63 & 36\\
      1 & 14 & 36 & 24
    \end{bmatrix},\quad
    \begin{bmatrix}
      1 & 4 & 6 & 4 & 1\\
      4 & 36 & 90 & 88 & 30\\
      6 & 90 & 306 & 372 & 150\\
      4 & 88 & 372 & 528 & 240\\
      1 & 30 & 150 & 240 & 120
    \end{bmatrix}&
    \\
    \begin{bmatrix}
      1
    \end{bmatrix},\quad
    \begin{bmatrix}
      0 & 1 \\
      1 & 2
    \end{bmatrix},\quad
    \begin{bmatrix}
      0 & 0 &1\\
      0 & 4 & 6\\
      1 & 6 &6
    \end{bmatrix},\quad
    \begin{bmatrix}
      0 & 0 & 0 & 1\\
      0 & 1 & 12 & 14\\
      0 & 12 & 45 & 36\\
      1 & 14 & 36 & 24
    \end{bmatrix},\quad
    \begin{bmatrix}
      0 & 0 & 0 & 0 & 1\\
      0 & 0 & 6 & 32 & 30\\
      0 & 6 & 90 & 228 & 150\\
      0 & 32 & 228 & 432 & 240\\
      1 & 30 & 150 & 240 & 120
    \end{bmatrix}&.
  \end{align*}
  \caption{Coefficient matrices of $\twocayhat_n(s,t)$,
    in the top row, and $\twocayhat^{\circ}_n(s,t)$, in the bottom row, for
    $n=1,2,3,4,5$. The $(i,j)$-th entry is the
    coefficient of $s^it^j$.}
\label{figure_joint_distr}
\end{figure} 

\begin{remark}
  A well known conjecture by Gessel asserts the $\gamma$-nonnegativity
  of the two-sided Eulerian polynomial. That is, for $n\ge 0$ there
  exist nonnegative integers $\gamma_{n,i,j}$, with $0\le i-1,j,j+2i\le n-1$,
  such that
  $$
  \twoeul_n(s,t)=\sum_{i,j}\gamma_{n,i,j}(st)^i(s+t)^j(1+st)^{n-1-j-2i}.
  $$
  For example, for $n=5$ we have
  $$
  \twoeul_5(s,t)=(1+st)^4+16(st)(1+st)^2+16(st)^2+6(st)(s+t)(1+st).
  $$
  This conjecture has recently been proved by Lin~\cite{L} via algebraic
  means. However, the $\gamma$-nonnegativity of the
  classical Eulerian polynomial $A_n(t)$ has a beautiful
  combinatorial proof by \emph{valley-hopping}~\cite{Pe} and it would be
  interesting to find a similar combinatorial proof for the two-sided
  case. In Theorem~\ref{Thm:twosided-caylerians}, we showed that
  \begin{align*}
  \twocay_n(s,t)
    &= \twoeul_n(1+s,1+t) \\
    &= \sum_{i,j}\gamma_{n,i,j}(st+s+t+1)^i(s+t+2)^j(st+s+t+2)^{n-1-j-2i}.
  \end{align*}
  Therefore, a combinatorial proof of the $\gamma$-nonnegativity could
  potentially be obtained by working with $\twocay_n(s,t)$, in terms of
  Burge words or Burge matrices. For instance, one could try to
  partition the set of Burge words (matrices) in classes whose
  contribution to $\twocay_n(s,t)$ is
  $(st+s+t+1)^i(s+t+2)^j(st+s+t+2)^{n-1-j-2i}$, where possibly $i$ and
  $j$ are some suitable parameters on $\Bur[n]$ (or $\Mat[n]$). This
  way, the coefficient $\gamma_{n,i,j}$ would count the number of such
  classes (with parameters $i$ and $j$), and thus would be nonnegative.
  In fact, it turns out that looking at binary matrices (or at the
  corresponding set of biwords) would be enough to prove the
  nonnegativity of the coefficients $\gamma_{n,i,j}$.  Indeed,
  by Theorem~\ref{Thm:twosided-caylerians},
  $$
  \twocay^{\circ}_n(s,t)=\twoeul^{\circ}_n(1+s,1+t).
  $$
  Let $\gamma^{\circ}_{n,i,j}$ be such that
  $$
  \twoeul^{\circ}_n(s,t)=\sum_{i,j}\gamma^{\circ}_{n,i,j}(st)^i(s+t)^j(1+st)^{n-1-j-2i}.
  $$
  Then
  \begin{align*}
    \twoeul^{\circ}_n(s,t)
    &= \sum_{p\in\Sym[n]}s^{\des(p)}t^{\asc(p)}\\
    &= \sum_{p\in\Sym[n]}s^{\des(p)}t^{n-1-\des(p)}\\
    &= t^{n-1}\twoeul_n\Bigl(s,\frac{1}{t}\Bigr)\\
    &= t^{n-1}\sum_{i,j}\gamma_{n,i,j}
      \Bigl(\frac{s}{t}\Bigr)^i\Bigl(s+\frac{1}{t}\Bigr)^j\Bigl(1+\frac{s}{t}\Bigr)^{n-1-j-2i}\\
    &= \sum_{i,j}\gamma_{n,i,j}s^it^{n-1-i-j-n+1+j+2i}(st+1)^j(t+s)^{n-1-j-2i}\\
    &=\sum_{i,j}\gamma_{n,i,j}(st)^i(s+t)^{n-1-j-2i}(1+st)^j.
  \end{align*}
  Finally,
  $$
  \sum_{i,j}\gamma^{\circ}_{n,i,j}(st)^i(s+t)^j(1+st)^{n-1-j-2i}
  =\sum_{i,j}\gamma_{n,i,j}(st)^i(s+t)^{n-1-j-2i}(1+st)^j
  $$
  from which
  $$
  \gamma^{\circ}_{n,i,j}=\gamma_{n,i,n-1-j-2i}.
  $$
  follows. Ultimately, the nonnegativity of the coefficients $\gamma_{n,i,j}$
  is expressed in terms of the strict two-sided Eulerian polynomial as
  $$
  \twocay^{\circ}_n(s,t)
  =\sum_{i,\ell}
  \gamma_{n,i,\ell}(st+s+t+1)^i(s+t+2)^{n-1-\ell-2i}(st+s+t+2)^{\ell},
  $$
  where $\ell=n-1-j-2i$.
%
\end{remark}

\section{Stanley's $v$-compatible maps}\label{section_comp_maps}

Let us now return to the Eulerian polynomial $A_n(t)$. Recall Carlitz
identity:
\begin{equation}
  \frac{tA_n(t)}{(1-t)^{n+1}}=\sum_{m \geq 1} m^nt^m.
\end{equation}
We wish to derive an analogous result for Cayley permutations.
Let us first sketch the proof presented by
Stanley~\cite[Proposition 1.4.4]{St} for the classical case.

\begin{definition}[Stanley~\cite{St}]\label{def_compatibility}
  Let $v\in\Sym[n]$. A map $x:[n]\to\nat$ is \emph{$v$-compatible} if it
  satisfies the following two conditions:
  \begin{itemize}
  \item[$(i)$] $x(v_1)\le x(v_2)\le\cdots\le x(v_n)$;
  \item[$(ii)$] $x(v_i)<x(v_{i+1})$ for each $i\in\Asc(v)$.
  \end{itemize}
\end{definition}
We remark that to better fit our presentation the roles of ascents and
descents, as well as the inequalities in $(i)$ and
$(ii)$, are reversed with respect to Stanley's exposition.

For $m\ge 0$, denote by $\comp_m(v)$ the set of $v$-compatible
maps $x:[n]\to[m]$.
Using a direct combinatorial argument, Stanley proves that
$$
|\comp_m(v)|=\multiset{m-\asc(v)}{n},
$$
where $\multiset{a}{b}=\binom{a+b-1}{b}$ is the number of multisets
of cardinality $b$ over a set of size $a$.
Stanley proceeds by showing that for each map $x:[n]\to[m]$ there is a unique
permutation $v$ of $[n]$ such that $x$ is $v$-compatible.
As a result, the disjoint union
$$
[m]^{[n]} = \bigcup_{v\in\Sym[n]}\comp_m(v),
$$
holds, which leads to the equation
\begin{equation}\label{eq_m_to_the_n}
m^n=\sum_{v\in\Sym[n]}|\comp_m(v)|=\sum_{v\in\Sym[n]}\multiset{m-\asc(v)}{n}.
\end{equation}
Now, the generating function for multisets over $[n]$, according
to size, is
\begin{equation}\label{eq_gf_multiset}
  \sum_{k\ge 0}\multiset{n}{k}t^k
  = (1+t+t^2+\cdots)^n \\
  = \frac{1}{(1-t)^n}.
\end{equation}
Moreover,
\begin{align*}
\frac{t^{1+\asc(v)}}{(1-t)^{n+1}}&=
\sum_{k\ge 0}\multiset{n+1}{k}t^{1+\asc(v)+k}\\
&=\sum_{k\ge 0}\binom{n+k}{k}t^{1+\asc(v)+k}\\
&=\sum_{k\ge 0}\binom{n+k}{n}t^{1+\asc(v)+k}\\
&=\sum_{m\ge 1}\binom{n+m-\asc(v)-1}{n}t^{m} & (m=1+\asc(v)+k)\\
&=\sum_{m\ge 1}\multiset{m-\asc(v)}{n}t^m\\
&=\sum_{m\ge 1}|\comp_m(v)|t^m
\end{align*}
and Carlitz identity follows:
\begin{align*}
  \frac{tA_n(t)}{(1-t)^{n+1}}
  &= \sum_{v\in\Sym[n]}\frac{t^{1+\asc(v)}}{(1-t)^{n+1}} \\
  &= \sum_{m\ge 1}\sum_{v\in\Sym[n]}|\comp_m(v)|t^m \\
  &= \sum_{m\ge 1} m^nt^m.
\end{align*}

We wish to use the Burge transpose to generalize the notion of
$v$-compatibility and obtain similar equations for the Caylerian
polynomials.
The relation between $v$-compatible maps and the Burge transpose
is highlighted in the next proposition.

\begin{proposition}\label{prop_comp_basis}
For $v\in\Sym[n]$ and $x:[n]\to[m]$,
$$
x\in\comp_m(v) \,\iff\, x\circ v \in\WIgen_m(v).
$$
In particular,
$$
\comp_m(v)=
\Gamma\bigl(\WIgen_m(v)\times\{ v\}\bigr)
\quand
|\comp_m(v)|=|\WIgen_m(v)|.
$$
\end{proposition}
\begin{proof}
It is easy to see that $x(v_i)<x(v_{i+1})$ for each $i\in\Asc(v)$
if and only if $\Des\bigl(x(v_1)\cdots x(v_n)\bigr)\subseteq\Des(v)$,
i.e.\ $x(v_1)\cdots x(v_n)\in\WIgen_m(v)$.
This proves the first statement.
Let us now prove the equality
$\comp_m(v)=\Gamma\left(\WIgen_m(v)\times\{ v\}\right)$.
If $x\in\comp_m(v)$, then
$\bigl(x(v_1)\cdots x(v_n),v\bigr)\in\PGenbur_{m}$ and
$$
\begin{pmatrix}
x(v_1)&\dots&x(v_n)\\v_1&\dots&v_n
\end{pmatrix}^T=
\begin{pmatrix}
1&\dots&n\\x(1)&\dots&x(n)
\end{pmatrix}
=\binom{\id}{x}.
$$
In particular, $x=\Gamma\bigl(x(v_1)\cdots x(v_n),v\bigr)$, as desired.
Conversely, let $x=\Gamma(u,v)$ for some $u\in\WIgen_m(v)$.
Then
$$
\begin{pmatrix}
u_1&\dots&u_n\\v_1&\dots&v_n
\end{pmatrix}^T=
\begin{pmatrix}
1&\dots&n\\x(1)&\dots&x(n)
\end{pmatrix}.
$$
Since the columns of the biword on the left-hand side are obtained
by flipping the right-hand side upside down
we have
$$
\binom{x(i)}{i}=\binom{x(v_j)}{v_j}=\binom{u_j}{v_j}
\quand
u_j=x(v_j),
$$
for each $i\in[n]$,
where $j=v^{-1}(i)$.
In particular, $\bigl(x(v_1)\cdots x(v_n),v\bigr)\in\PBur_{m}$
and $x\in\comp_m(v)$, which gives the desired equality.
Finally, $|\comp_m(v)|=|\WIgen_m(v)|$
is an immediate consequence of $\Gamma$ being
injective on $\WIgen_m(v)\times\{ v\}$.
\end{proof}

Stanley showed that for each map $x:[n]\to[m]$ there
is a unique permutation $v\in\Sym[n]$ such that $x$ is $v$-compatible.
Let us rederive this result using our machinery.
By Proposition~\ref{prop_comp_basis}, $x$ is $v$-compatible
if and only if $x=\Gamma(u,v)$, where $u=\sort(x)$ and
$(u,v)\in\Genbur_{m}[n]$. Thus, if $x$ is $v$-compatible, we have
$(u,v)^T=(\id,x)$,
or, equivalently, $(\id,x)^T=(\sort(x),v)$
and $v$ is uniquely determined as
$$
  v=\Gamma(\id,x).
$$

Let $\comp(v)=\bigcup_{m\ge 0}\comp_m(v)$.
\begin{corollary}\label{cor_fish_basis}
  The Fishburn basis of a permutation $v$ is the set of
  $v$-compatible maps that are Cayley permutations:
  $$
  \basis(v)=\Cay\cap\comp(v).
  $$
\end{corollary}
\begin{proof}
  Assume $v\in\Sym$ and let $\WIgen(v)=\bigcup_{m\ge 1}\WIgen_m(v)$.
  By definition,
  $$
  \comp(v)=
  \Gamma\bigl(\WIgen(v)\times\{ v\}\bigr)
  \quand
  \basis(v)=\Gamma\bigl(\WI(v)\times\{ v\}\bigr).
  $$
  Clearly, $\WI(v)=\WIgen(v)\cap\Cay$ and thus
  \begin{align*}
    \comp(v)\cap\Cay
    &= \Gamma\Bigl(\WIgen(v)\times\{ v\}\Bigr)\cap\Cay\\
    &= \Gamma\Bigl(\bigl(\WIgen(v)\cap\Cay\bigr)\times\{ v\}\Bigr)\\
    &= \Gamma\Bigl(\WI(v)\times\{ v\}\Bigr)\\
    &= \basis(v).\qedhere
  \end{align*}
\end{proof}

Now, it is not immediately clear how to extend Stanley's notion of
$v$-compatibility to Cayley permutations. In fact, if $v\in\Cay[n]$ and
$x:[n]\to[m]$, then the presence of repeated entries in $v$ makes
the inequalities $\mathrm{(i)}$ and $\mathrm{(ii)}$ in
Definition~\ref{def_compatibility} meaningless. On the other hand, we
showed that the set of $v$-compatible maps is characterized
by $\comp_m(v)=\Gamma\left(\WIgen_m(v)\times\{ v\}\right)$,
and $\WIgen_m(v)$ is defined for any Cayley permutation $v$.
We shall generalize the definition of \emph{$v$-compatible maps}
accordingly by setting, for any $v\in\Cay$ and $m\ge 0$,
$$
\comp_m(v)=
\Gamma\bigl(\WIgen_m(v)\times\{ v\}\bigr).
$$
Furthermore, we define the set of \emph{strictly $v$-compatible maps} by
$$
\compbin_m(v)=
\Gamma\bigl(\WIgenbin_m(v)\times\{ v\}\bigr).
$$
By Proposition~\ref{prop_comp_basis}, our notion of
$v$-compatibility matches Stanley's definition when $v\in\Sym$.
The same goes for strict $v$-compatibility: if $v\in\Sym$, then
$\Des(v)=\Des^{\circ}(v)$ and $u\in\WIgen_m$ if and only if
$u\in\WIgenbin_m$.

\begin{example}
Let us give an example to illustrate the notion of
$v$-compatibility. Recall that
$\WIgen_m(v)=\{ u\in\WIgen_m[n]: \Des(u)\subseteq\Des(v)\}$, where $n$ is the length of $v$.
To compute $\WIgen_m(v)$, consider a biword
$$
\binom{u}{v}=
\begin{pmatrix}
u(1)&\dots&u(n)\\v(1)&\dots&v(n)
\end{pmatrix},
$$
where $u:[n]\to[m]$ is weakly increasing. Note that
$\Des(u)\subseteq\Des(v)$ if and only if $\Asc^{\circ}(v)\subseteq\Asc^{\circ}(u)$;
that is, among all the weakly increasing maps, to obtain $\WIgen_m(v)$
we will pick only those $u$ where $u(i+1)>u(i)$ for each $i\in\Asc^{\circ}(v)$.
Similarly, the set $\WIgenbin_m(v)$ consists of the weakly increasing
maps $u:[n]\to[m]$ that satisfy the stricter requirement that
$u(i+1)>u(i)$ for each $i\in\Asc(v)$.
For instance, let $v=331412$ and $m=4$. Here $n=6$,
$\Des(v)=\{ 1,2,4\}$ and $\Des^{\circ}(v)=\{ 2,4\}$. 
Thus,
$$
\WIgen_4(v)=
\{ 111223,111224,111234,111334,112334,122334,222334\}
$$
and
$
\WIgenbin_4(v)=\{
122334
\}.
$
The sets of $v$-compatible and strictly $v$-compatible maps
are obtained by applying $\Gamma$ to the biwords in
$\comp_m(v)\times\{ v\}$ and
$\compbin_m(v)\times\{ v\}$, respectively.
We get
$$
\comp_4(v)=\{
213112,214112,314112,314113,324113,324213,324223
\},
$$
where for instance $214112=\Gamma(111224,331412)$ since
$\displaystyle{\binom{111224}{331412}^{\!T}\! = \binom{112334}{214112}}$.
Also,
\[
  \compbin_4(v)=\{ \Gamma(122334,v)\}=\{ 324213\}.
\]
\end{example}

Next we wish to count compatible and strictly compatible maps.

\begin{proposition}\label{prop_num_of_compat}
  Let $v\in\Cay[n]$ and $m\ge 0$. Then
  $$
  |\comp_m(v)|=\multiset{m-\asc^{\circ}(v)}{n}
  \quand
  |\compbin_m(v)|=\multiset{m-\asc(v)}{n}.
  $$
\end{proposition}
\begin{proof}
  We shall focus on the first identity. Let $v\in\Cay[n]$ and
  $a=\asc^{\circ}(v)$.  Recall that
  $\comp_m(v)=\Gamma\bigl(\WIgen_m(v)\times\{ v\}\bigr)$.  Since
  $\Gamma$ is an injective map on $\WIgen_m(v)\times\{ v\}$, it suffices
  to show that $|\WIgen_m(v)|=\multiset{m-a}{n}$.  Let $Y=Y_m(v)$ be the
  set of strings with letters ``$\star$'' and ``$\mid$'' that contain exactly
  $n$ stars and $m-a-1$ bars. Clearly,
  $$
  |Y|=\binom{n+m-a-1}{n}=\multiset{m-a}{n}.
  $$
  We shall construct a bijection between $\WIgen_m(v)$ and $Y$.  To each
  string $y\in Y$, we associate a weakly increasing map $u$ of length
  $n$ by setting, for $i\in[n]$,
  $$
  u(i)=a_i+b_i+1,
  $$
  where
  $a_i=|\{ j\in[i-1]:v(j)<v(j+1)\}|$ is the number
  of strict ascents preceding the $i$th entry of $v$, and
  $b_i$ is the number of bars preceding the $i$th star of $y$.

  By definition, $u$ is weakly increasing and
  \begin{align*}
    \max(u) = u(n)
    &= a_n+b_n+1 \\
    &\le a+(m-a-1)+1=m.
  \end{align*}
  Thus $u\in\WIgen_m$. Furthermore, if $i\in\Asc^{\circ}(v)$, then $a_{i+1}=a_i+1$ and hence
  \begin{align*}
    u(i+1) &= a_{i+1}+b_{i+1}+1 \\
           &= a_i+1+b_{i+1}+1 \\
           &\ge a_i+1+b_i+1 = u(i)+1.
  \end{align*}
  In other words, $\Asc^{\circ}(v)\subseteq\Asc^{\circ}(u)$. This is of course
  equivalent to $\Des(u)\subseteq\Des(v)$, i.e.\ $u\in\WIgen_m(v)$, as
  wanted. On the other hand, every $u\in\WIgen_m(v)$ gives rise to a
  string $y\in Y$, where the $i$th star is preceded by $u(i)-a_i-1$
  bars.  Therefore, the correspondence between $Y$ and $\comp_m(v)$ is
  bijective, which completes the proof of the first equality. We omit
  the proof of the second equality since it is obtained by a
  straightforward modification of the proof just given.
\end{proof}

\newcommand{\starr}{\makebox[1.3ex]{\ensuremath{\star}}}
\newcommand{\barr}{\makebox[1ex]{\ensuremath{\mid}}} Let us give an
example illustrating the bijection between $Y_m(v)$ and $\comp_m(v)$ in
the proof of Proposition~\ref{prop_num_of_compat}.  Let $m=9$, $n=6$ and
$v=221312\in\Cay[n]$.  Note that $\Des(v)=\{ 1,2,4\}$,
$\Asc^{\circ}(v)=\{ 3,5\}$ and $a_1a_2\dots a_n=000112$. Consider the string
\[
  y\,=\,\barr\starr\starr\barr\starr\barr\barr\starr\starr\barr\starr\barr.
\]
We have $b_1b_2\dots b_n=112445$ and calculating $u(i)=a_i+b_i+1$, for
each $i\in[n]$, we find that the map $u:[n]\to[m]$ associated with $y$
is $u=223668$. Note that $\Des(u)=\{ 1,4\}\subseteq\Des(v)=\{ 1,2,4\}$
and hence $u\in \WIgen_m(v)$.

\begin{theorem}\label{Thm:eqs-cayler}
For each $n\ge 0$,
\begin{itemize}
\item[$(i)$] $\,\displaystyle{\frac{tC_n(t)}{(1-t)^{n+1}}=
\sum_{m\ge 1}|\Genbur_{m}[n]|t^m};$
\item[$(ii)$] $\,\displaystyle{\frac{tC^{\circ}_n(t)}{(1-t)^{n+1}}=
\sum_{m\ge 1}|\BGenbur_{m}[n]|t^m};$
\item[$(iii)$] $\,\displaystyle{\frac{tA_n(t)}{(1-t)^{n+1}}=
\sum_{m\ge 1}|\PGenbur_{m}[n]|t^m}.$
\end{itemize}
\end{theorem}
\begin{proof}
Let $a\in\nat$. Using equation~\eqref{eq_gf_multiset} we obtain
$$
\frac{t^{a+1}}{(1-t)^{n+1}}=
\sum_{k\ge 0}\multiset{n+1}{k}t^kt^{a+1}=
\sum_{m\ge 1}\multiset{m-a}{n}t^m,
$$
where the last equality follows by setting $m=k+a+1$ and noting
that
$\multiset{n+1}{k}=
\binom{n+k}{n}=
\multiset{m-a}{n}$.
Now, by using Proposition~\ref{prop_num_of_compat} and setting
$a=\asc^{\circ}(v)$ in the above equation, we get
\begin{align*}
\frac{tC_n(t)}{(1-t)^{n+1}}&=
\sum_{v\in\Cay[n]}\frac{t^{\asc^{\circ}(v)+1}}{(1-t)^{n+1}}\\
&=\sum_{v\in\Cay[n]}\sum_{m\ge 1}\multiset{m-\asc^{\circ}(v)}{n}t^m\\
&=\sum_{m\ge 1}\left(\sum_{v\in\Cay[n]}|\comp_m(v)|\right)t^m\\
&=\sum_{m\ge 1}\left(\sum_{v\in\Cay[n]}|\WIgen_m(v)|\right)t^m\\
&=\sum_{m\ge 1}|\Genbur_{m}[n]|t^m,
\end{align*}
where the last equality follows from the disjoint union
$$
\Genbur_{m}[n]=\bigcup_{v\in\Cay[n]}\WIgen_m(v)\times\{ v\}.
$$
The equation for $C^{\circ}_n(t)$ can analogously be obtained by setting
$a=\asc(v)$ instead of $a=\asc^{\circ}(v)$, and the equation for $A_n(t)$
follows in a similar manner.
\end{proof}

\section{Cayley permutations with a prescribed ascent set}\label{section_alphas}


Let $n\ge 1$ and consider a subset $S\subseteq[n-1]$ of size $r$, say
$S=\lbrace s_1,\dots,s_r\rbrace$, with $s_1<s_2<\dots<s_r$. Define
$$
\alphas=|\lbrace v\in\Sym[n]:\Asc(v)\subseteq S\rbrace|.
$$
It is well known (see e.g.\ Stanley~\cite{St}, Proposition 1.4.1) that
$\alphas$ is a polynomial in $n$ given by the multinomial coefficient
\begin{align*}
\alphas
  &= \binom{n}{s_1,s_2-s_1,s_3-s_2,\dots,n-s_r} \\
  &= \frac{n!}{s_1!(s_2-s_1)!(s_3-s_2)!\cdots(n-s_r)!}.
\end{align*}
Indeed, any permutation $v\in\Sym[n]$ such that $\Asc(v)\subseteq S$
is obtained by first choosing $s_1$ elements $v_1>v_2>\dots>v_{s_1}$,
then $s_2-s_1$ elements $v_{s_1+1}>v_{s_1+2}>\dots>v_{s_2}$, and so on.
In other words, $\alphas$ equals the number of
ballots on $[n]$ whose block sizes are given by $s_1$, $s_2-s_1$, \dots, $n-s_r$.

Transitioning from permutations to Cayley permutations---the theme of
this article--- we define
\begin{align*}
\betasw &=|\lbrace v\in\Cay[n]:\, \Asc(v)\subseteq S\rbrace|;\\
\betass &=|\lbrace v\in\Cay[n]:\, \Asc^{\circ}(v)\subseteq S\rbrace|.
\end{align*}
We wish to show that Burge matrices and Burge words provide an answer
to the following natural question:

\begin{center}
\emph{What is the combinatorial structure we should replace ballots with,
so that $\betasw$ and $\betass$ equals the number of structures on $[n]$
with block sizes $s_1$, $s_2-s_1$, \dots, $n-s_r$?}
\end{center}

Given $v\in\Cay[n]$, consider the biword
$$
\binom{\eta(S)}{v}=
\begin{pmatrix}
1\dots 1         & 2\dots 2               & \dots & r+1\dots r+1\\
v_1\dots v_{s_1} & v_{s_1+1}\dots v_{s_2} & \dots & v_{s_r+1}\dots v_n
\end{pmatrix},
$$
where
$$
\eta(S)=1^{s_1}2^{s_2-s_1}3^{s_3-s_2}\cdots r^{s_r-s_{r-1}}(r+1)^{n-s_r}
$$
is the unique weakly increasing Cayley permutation of length $n$
whose descent set is $\Des\bigl(\eta(S)\bigr)=[n-1]\setminus S$.
We will refer to the entries in $v$ that lie below the integer $i$ as
the $i$th \emph{block} of $v$ in the ballot induced by $S$.

\begin{lemma}\label{lemma_etaS}
For any $v\in\Cay[n]$,
$$
\Asc^{\circ}(v)\subseteq S\iff\binom{\eta(S)}{v}\in\Bur
\quad\text{and}\quad
\Asc(v)\subseteq S\iff\binom{\eta(S)}{v}\in\BBur.
$$
\end{lemma}
\begin{proof}
To prove the first statement, observe that $\Asc^{\circ}(v)\subseteq S$ if and
only if $[n-1]\setminus S\subseteq\Des(v)$, which is equivalent to
$(\eta(S),v)\in\Bur$ since $\Des(\eta(S))=[n-1]\setminus S$.
In other words, we have just showed that $\Asc^{\circ}(v)\subseteq S$ if
and only if each block of $v$ in the ballot induced by $S$ is weakly
decreasing, which in turn is the same as $(\eta(S),v)\in\Bur$.
The second statement can be proved in a similar fashion. In this case,
$\Asc(v)\subseteq S$ if and only if each block of $v$ is strictly
decreasing, that is, $\Des(\eta(S))\subseteq\Des^{\circ}(v)$ and 
$(\eta(S),v)\in\BBur$.
\end{proof}

Now, define
$$
\Bur(S)=
\left\lbrace\binom{\eta(S)}{v}:v\in\Cay\right\rbrace \cap \Bur.
$$
Similarly, define $\BBur(S) =\Bur(S)\cap\BBur$ and
$\PBur(S)=\Bur(S)\cap\PBur$.
It is straightforward to describe the sets of Burge matrices
corresponding to $\Bur(S)$, $\BBur(S)$ and $\PBur(S)$:
Recall that, if $A=\map^{-1}(u,v)$ is the Burge matrix corresponding
to $(u,v)\in\Bur$, then $A$ is the $\max(u)\times\max(v)$
matrix whose entry $a_{ij}$ is equal to the number of columns
$\binom{i}{j}$ contained in $(u,v)$. Now, if $u=\eta(S)$, then
the first row of $A$ sums to $s_1$, the second row sums to $s_2-s_1$, and
so on up the $(r+1)$th (and last) row, whose sum is $n-s_r$.
To express this in a more compact form, let $\ones$ denote the all ones
vector so that $\hproj{A}$ is the \emph{row sum vector}
of $A$, i.e.\ the vector whose $i$th entry is equal
to the sum of the $i$th row of $A$.
We then define
\begin{align*}
\Mat(S)&=\lbrace A\in\Mat:\hproj{A}=\Delta(S)\rbrace,\\
\intertext{where $\Delta(S) = (s_1,s_2-s_1,\dots,s_r-s_{r-1},n-s_r)$. We further define}
\BMat(S)&=\Mat(S)\cap\BBur(S);\\
\PMat(S)&=\Mat(S)\cap\PBur(S).
\end{align*}
In analogy with Proposition~\ref{prop_map_subsets}, we have
$$
\map\bigl(\Mat(S)\bigr)=\Bur(S);\quad
\map\bigl(\BMat(S)\bigr)=\BBur(S);\quad
\map\bigl(\PMat(S)\bigr)=\PBur(S).
$$
Let $\Bur(S)[n] = \Bur(S)\cap \Bur[n]$, $\Mat(S)[n]=\Mat(S)\cap\Mat[n]$, etc.
The following proposition is an immediate consequence of Lemma~\ref{lemma_etaS}.

\begin{theorem}\label{Thm:prescr-ascset}
For each $n\ge 1$ and $S\subseteq[n-1]$,
\begin{itemize}
\item[$(i)$] $\betass=|\Bur(S)[n]|=|\Mat(S)[n]|;$
\item[$(ii)$] $\betasw=|\BBur(S)[n]|=|\BMat(S)[n]|;$
\item[$(iii)$] $\alphas=|\PBur(S)[n]|=|\PMat(S)[n]|.$
\end{itemize}
\end{theorem}

Arguably, we have now established that Burge matrices of size $n$ with
fixed row sums,
i.e.\ $\Mat(S)[n]$ and $\BMat(S)[n]$, and Burge words of size $n$ whose
top row is equal to $\eta(S)$, i.e.\ $\Bur(S)[n]$ and $\BBur(S)[n]$,
are combinatorial structures that generalize ballots so that
$\betass$ and $\betasw$ count the number of structures on $[n]$
whose ``block sizes'' are given by $S\subseteq[n-1]$.

Two expressions relating $\betass$ and $\betasw$ to permutations follow from
Lemma~\ref{lemma_basis_descents}:
\begin{align*}
  \betass &=\sum_{\substack{v\,\in\,\Sym[n]\\ \Asc(v)\,\subseteq\, S}}2^{\des(v^{-1})};\\[1ex]
  \betasw &=\sum_{\substack{v\,\in\,\Sym[n]\\ \Des(v)\,\subseteq\, S}}2^{\des(v^{-1})}.
\end{align*}
Indeed, for any $w\in\Sym$ and Cayley permutation $y\in\basis(w)$ we have
$\Des(y)=\Des(w^{-1})$ and $\Asc^{\circ}(y)=\Asc(w^{-1})$. Furthermore,
$|\basis(w)|=2^{\des(w)}$, $\Cay[n]=\bigcup_{w\in\Sym[n]}\basis(w)$
and the desired identities follow by setting $v=w^{-1}$.

\section{Future work\label{section_final}}

In Theorem~\ref{Thm:twosided-caylerians}, we showed that
$\twoeul_n(1+s,1+t)=\twocay_n(s,t)$.
In other words, by marking descents and inverse descents
of permutations with $1+s$ and $1+t$, respectively,
we obtain Burge matrices $A$ with weight $s^{n-\rows(A)}t^{n-\cols(A)}$;
or, equivalently, Burge words $(u,v)$ with weight $s^{n-\max(u)}t^{n-\max(v)}$.
Going one step further, what structure arises from $\twocay_n(1+s,1+t)$?

Working with generating functions it is easy to show that
\[
  A_n(-1)=
  \begin{cases}
    (-1)^{(n-1)/2} E_n &\text{if $n$ is odd,} \\
    0 &\text{if $n$ is even,}
  \end{cases}
\]
where $E_n$ is the $n$th Euler number. Can something interesting be said
about the Caylerian polynomials evaluated at $-1$?  Due to the symmetry
$C^{\circ}_n(t)=t^{n-1}C_n(1/t)$, we have $C^{\circ}_n(-1)=(-1)^{n-1}C_n(-1)$.
The sequence $C_n(-1)$, $n\ge 0$, begins (see also A365449 in the
OEIS~\cite{Sl})
$$
1,1,-1,-3,11,45,-301,-1475,14755,85629.
$$

As mentioned in Section~\ref{section_alphas}, the number $\alphas$ of
permutations of size $n$ whose ascent set is contained in $S$ is given
by a polynomial in $n$. The corresponding property is false for
$\betass$. In particular, $\kappa^{\circ}_n(\emptyset)=2^{n-1}$ for $n\geq 1$.
We have, however, gathered some numerical evidence that suggests that
$\betasw$ is a polynomial and we make the following conjecture (in which
$\max(\emptyset)=0$ by convention).

\begin{conjecture}
  Let $S$ be any finite subset of positive integers and let
  $k=\max(S)$. Then the sequence $n\mapsto\kappa_{n+k}\bigl(S\bigr)$ is
  a polynomial in $n$ of degree $k$.
\end{conjecture}

For instance, the polynomials for $S\subseteq \{1,2,3\}$  are conjectured to be
\begin{gather*}
\kappa_n\bigl(\emptyset\bigr) = 1;\quad
\kappa_{n+1}\bigl(\{1\}\bigr) = 1 + 2 n;\quad
\kappa_{n+2}\bigl(\{2\}\bigr) = 1 + 2 n + 2 n^{2}; \\[1ex]
\kappa_{n+3}\bigl(\{3\}\bigr) = 1 + \frac{8}{3} n + 2 n^{2} + \frac{4}{3} n^{3};\quad
\kappa_{n+2}\bigl(\{1, 2\}\bigr) = 3 + 6 n + 4 n^{2}; \\[1ex]
\kappa_{n+3}\bigl(\{1, 3\}\bigr) =
\kappa_{n+3}\bigl(\{2, 3\}\bigr) = 5 + 12 n + 10 n^{2} + 4 n^{3}; \\[1ex]
\kappa_{n+3}\bigl(\{1, 2, 3\}\bigr) = 13 + 30 n + 24 n^{2} + 8 n^{3}.
\end{gather*}

There is a nice determinant formula for the number,
\[
  \beta_n(S)=|\lbrace v\in\Sym[n]:\Asc(v)= S\rbrace|,
\]
of permutations of $[n]$ whose ascent set is equal to $S$, namely
\[
  \beta_n(S)=n!\det\bigl[1/(s_{j}-s_{i-1})\bigr],
\] where $S=\{s_1,s_2,\dots,s_k\}$ and $(i,j)\in[k+1]\times [k+1]$.  See
Stanley~\cite[p.\ 229]{St} for further details. Is there a nice formula
for the number of Cayley permutations whose ascent set is $S$?



One way to generalize the number $\alphas$ to a polynomial in $t$ is to let
\[
\alpha_n(S;t) =
  \sum_{\substack{v\in\Sym[n]\\ \Asc(v)\subseteq S}} t^{\des(v^{-1})}.
\]
for $S\subseteq [n-1]$. Then $\alpha_n(S;1)=\alphas$ and
$\betasw=\alpha_n(S;2)$. What can we more generally say about
$\alpha_n(S;t)$?

In this paper, we have obtained a variety of results linking the
(two-sided) Caylerian polynomials, as well as the coefficients $\betasw$
and $\betass$, to certain sets of Burge words and Burge matrices.  An
in-depth study of the enumerative properties of the latter
structures will be carried out in a forthcoming paper.

\end{document}